\documentclass[10pt]{amsart}

\usepackage{eucal,url,amssymb,stmaryrd,enumerate,amscd,}
\usepackage[pagebackref,colorlinks=true]{hyperref}
\usepackage{amsfonts}
\usepackage{amsmath,amsthm,amssymb,amscd,enumerate,eucal,url,stmaryrd,verbatim}

\numberwithin{equation}{section}
\newtheorem{thrm}{Theorem}[section]
\newtheorem{lemma}[thrm]{Lemma}
\newtheorem{prop}[thrm]{Proposition}
\newtheorem{cor}[thrm]{Corollary}
\newtheorem{dfn}[thrm]{Definition}

\newtheorem{rmrk}[thrm]{Remark}

\newtheorem{conv}[thrm]{Convention}
\newcommand{\f}{\omega}
\newcommand{\C}{\nabla}

\newcommand{\be}{\begin{equation}}
\newcommand{\ee}{\end{equation}}
\newcommand{\ta}{T^a}
\newcommand{\tb}{T^{sym}}

\newcommand{\dx}{\partial/\partial x}
\newcommand{\dy}{\partial/\partial y}
\newcommand{\dz}{\partial/\partial z}

\newcommand{\dta}{\partial/\partial t_{a}}
\newcommand{\dxa}{\partial/\partial x_{a}}
\newcommand{\dya}{\partial/\partial y_{a}}
\newcommand{\dza}{\partial/\partial z_{a}}

\title[Conformal paraquaternionic contact curvature and the local flatness theorem]
{Conformal paraquaternionic contact curvature and the local flatness theorem}

\begin{document}

\begin{abstract}
A tensor invariant is defined on a paraquaternionic contact  manifold in terms of the
curvature and torsion of the canonical paraquaternionic connection involving  derivatives up to
 third order of the contact form. This tensor,
called paraquaternionic contact conformal curvature, is similar to
the Weyl conformal curvature in Riemannian geometry, the
Chern-Moser tensor in CR geometry, the para contact curvature in para CR geometry and to the quaternionic contact conformal curvature in quaternionic contact geometry.

It is shown that a paraquaternionic contact
manifold is locally paraquaternionic contact conformal to the standard flat paraquaternionic contact structure
on the paraquaternionic Heisenberg group, or equivalently, to the
standard para 3-Sasakian structure on the paraquaternionic pseudo-sphere iff the paraquaternionic contact
conformal curvature vanishes.


\end{abstract}


\date{\today}

 \author{Stefan Ivanov}
\address[Stefan Ivanov]{University of Sofia, Faculty of Mathematics and Informatics,
blvd. James Bourchier 5, 1164, Sofia, Bulgaria}
\address{and Institute of Mathematics and Informatics, Bulgarian Academy of
Sciences} \email{ivanovsp@fmi.uni-sofia.bg}

\author{Marina Tchomakova}
\address[Marina Tchomakova]
{University of Sofia, Faculty of Mathematics and Informatics,
blvd. James Bourchier 5, 1164, Sofia, Bulgaria} \email{mari\_06@abv.bg}

\author{Simeon Zamkovoy}

\maketitle

\tableofcontents

\setcounter{tocdepth}{2}

\section{Introduction}

We develop  a tensor invariant in  the sub-Riemannian geometry of para contact structures on a $4n+3$-dimensional differentiable manifold related to  the algebra of para quaternions, known also as split quaternions \cite{Swann}, quaternions of the second kind \cite{L}, and complex product structures \cite{AS}. The paraquaternionic contact structures,  introduced in \cite{CIZ},  turns out to be a  generalization of the para 3-Sasakian geometry developed in \cite{AK0,Swann}.

Paraquaternionic contact  geometry is a topic with some analogies  with the quaternionic contact  geometry introduced by O.Biquard \cite{B} and its developments in connection with finding  the extremals and the best constant in the $L^2$ Folland-Stein inequality on the quaternionic Heisenberg group and related quaternionic contact Yamabe problem \cite{IMV1,IMV2,IP,IMV3,IV}, but  also with differences mainly because the paraquaternionic contact structure lead to sub-hyperbolic PDE instead of sub-elliptic PDE in the quaternionic contact case. 


A paraquaternionic contact (pqc) manifold $(M, [g], \mathbb{PQ})$ is a 
$4n+3$-dimensional manifold $M$ with a codimension three
distribution $H$  locally given as the kernel of a 1-form
$\eta=(\eta_1,\eta_2,\eta_3)$ with values in $\mathbb{R}^3$. In addition, $H$ has a conformal  $Sp(n,\mathbb R)Sp(1,\mathbb R)$ structure, i.e. it is
equiped with a conformal class of neutral  metrics $[g]$ of signature $(2n,2n)$ and a rank-three bundle
$\mathbb {PQ}$ consisting of (1,1)-tensors on $H$, locally generated
by two almost para complex structures $I_1,I_2$ and an almost complex structure $I_3$ on $H$, satisfying
the identities of the imaginary unit para quaternions,
$I_1^2=I_2^2=id_{H},\quad I_3^2=-id_H,\quad I_1I_2=-I_2I_1=I_3,$ such that 
$-\epsilon_i2g(I_iX,Y)\ =\ d\eta_i(X,Y), \quad \epsilon_1=\epsilon_2=-\epsilon_3=1,\quad g\in[g].$ 

The 1-form $\eta$ is determined up to a conformal factor. Hence,
$H$ becomes equipped with a conformal class [g] of neutral
Riemannian metrics of signature (2n,2n). Transformations preserving a
given pqc structure $\eta$, i.e.
$\bar\eta=\mu\Phi\eta$ for a non-vanishing smooth function $\mu$ and a $SO(1,2)$ valued  smooth matrix $\Phi$, are
called \emph{paraquaternionic  contact conformal (pqc conformal) transformations}. To every metric in the fixed conformal class one can
associate a linear connection $\nabla$ preserving the pqc structure, introduced in  \cite{CIZ},  the (canonical)  pqc connection.

A basic example
is provided by any para 3-Sasakian manifold, which can be defined as a $(4n+3)$%
-dimensional pseudo Riemannian manifold, whose Riemannian cone is a hypersymplectic manifold. 


The paraquaternionic Heisenberg group $pQH$ with its "standard"
left-invariant pqc structure is the unique (up to a $SO(1,2)$-action)
example of a pqc structure with flat canonical connection \cite{CIZ}.
{As a manifold $pQH \ = pH^n\times\text {Im}\, pH$,
while the group multiplication is given by
\[
( q', \omega')\ =\ (q_o, \omega_o)\circ(q, \omega)\ =\ (q_o\ +\ q, \omega\ +\ \omega_o\ +
\ 2\ \text {Im}\  q_o\, \bar q),
\]
\noindent where $q,\ q_o\in pH^n$ and $\omega, \omega_o\in \text
{Im}\, pH$. The standard flat paraquaternionic contact structure
is defined by the left-invariant paraquaternionic contact form
\[\tilde\Theta\ =\ (\tilde\Theta_1,\ \tilde\Theta_2, \
\tilde\Theta_3)\ =\ \frac 12\ (d\omega \ - \ q \cdot d\bar q \ + \
dq\, \cdot\bar q),\] where $.$ denotes the para quaternion
multiplication.}

The aim of this paper is to find a tensor invariant on the tangent
bundle, characterizing locally the pqc structures, which are
paraquaternionic contact conformally equivalent to the flat
pqc structure on the  paraquaternionic Heisenberg group $pQH$ . With this goal in mind, we describe a curvature-type
tensor $W^{pqc}$ defined in terms of the curvature and torsion of
the canonical connection by \eqref{qccm}, involving  derivatives up to
second order of the horizontal metric, whose form is
similar to the Weyl conformal curvature in Riemannian geometry
(see e.g. \cite{Eis}), to the Chern-Moser invariant in CR
geometry \cite{ChM}, see also \cite{W}, the para contact conformal curvature developed in \cite{IVZ} and to the quaternionic contact conformal curvature  described by Ivanov-Vassilev in \cite{IV1}, see also \cite{IV}. We call $W^{pqc}$ the
\emph{paraquaternionic contact conformal curvature}, or \emph{pqc conformal
curvature}. The main purpose of this article is to prove
the following two facts.
\begin{thrm}\label{main1}
The pqc conformal curvature $W^{pqc}$ is invariant under pqc conformal transformations.
\end{thrm}
\begin{thrm}\label{main2}
A pqc structure on a (4n+3)-dimensional smooth pqc manifold is locally pqc conformal to the
standard flat pqc structure on the paraquaternionic Heisenberg group $pQH$
if and only if the pqc conformal curvature vanishes,
$W^{pqc}=0$.
\end{thrm}
We define a local map between the paraquaternionic Heisenberg group and the para 3-Sasakian pseudo-sphere (called the paraquaternionic Cayley transform) and show in Proposition~\ref{phps} below that the paraquaternionic Cayley transform  
 establishes  a conformal paraquaternionic
contact automorphism between the standard para 3-Sasaki structure on the
paraquaternionic pseudo-sphere $pS^{4n+3}$ and the standard flat  pqc structure on $pQH$.  As a consequence of Theorem~\ref{main2} and Proposition~\ref{phps}, we obtain 
\begin{cor}\label{main3}
A pQC manifold is locally paraquaternionic contact conformal to the paraquaternionic pseudosphere
$pS^{4n+3}$ if and only if the pqc conformal curvature vanishes,
$W^{pqc}=0$.
\end{cor}

Our investigations follow the classical approach used by
H.Weyl, see e.g. \cite{Eis} and are  close to \cite{IVZ} and \cite{IV1} while  \cite{ChM} and \cite{W}  follow the Cartan method of equivalence.

\begin{rmrk}
Following the work of Cartan and Tanaka, a pqc structure can be
considered as an example of what has become known as a parabolic
geometry. The paraquaternionic Heisenberg group, as well as the
paraquaternionic pseudo-sphere, provide the flat
models of such a geometry due to the paraquaternionic Cayley transform.
 It is well known that the curvature of
the corresponding {regular} Cartan connection is the obstruction for
the local flatness. However, the Cartan curvature is not a tensor
field on the tangent bundle and it is highly nontrivial to extract a
tensor field involving the lowest order derivatives of the structure,
which implies the vanishing of the obstruction.
 Theorem 1.2 suggests that
a necessary and sufficient condition for the vanishing of the
Cartan curvature of a pqc structure is the vanishing of the
pqc conformal curvature tensor, $W^{pqc}=0$.
\end{rmrk}

\textbf{Organization of the paper.} The paper relies heavily on  the canonical connection introduced in
\cite{CIZ} and the properties of its torsion and curvature described in \cite{CIZ}.
In order to make the present paper self-contained, in Section \ref{s:review}
we give a review of the notion of a paraquaternionic contact structure and
collect formulas and results from \cite{CIZ}  that will be used in the
subsequent sections.
\begin{conv}\label{conven}
We use the following conventions:
\begin{enumerate}[a)]
\item We shall use $X,Y,Z,U$ to denote horizontal vector fields, i.e. $X,Y,Z,U\in H$;
\item $\{e_1,\dots,e_{n},I_1e_1,\dots,I_1e_{n},Ie_2,\dots,I_2e_{n},I_3e_1,\dots,I_3e_{n}\}$ denotes an adapted
orthonormal basis of the horizontal space $\mathbb H$.;
\item The summation convention over repeated vectors from the basis
$\{e_1,\dots,e_{4n}\}$ will be used,
\begin{multline*}
P(e_b,e_b)=\sum_{b=1}^{4n}g(e_b,e_b)R(e_b,e_b)=\sum_{b=1}^{n}\Big[P(e_b,e_b)-P(I_1e_b,I_1e_b)
-P(I_2e_b,I_2e_b)+P(I_3e_b,I_3e_b)\Big];
\end{multline*}
\item The triple $(i,j,k)$ denotes any cyclic permutation of
$(1,2,3)$. In particular, any equation involving $i,j,k$ holds for
any such permutation;
\item $s$ and $t$ will be any numbers from the set $\{1,2,3\}, \quad
s,t\in\{1,2,3\}$.
\end{enumerate}
\end{conv}

\textbf{Acknowledgements}  The research of M. Tch.  is partially  financed by the European Union-Next Generation EU, through the National Recovery and Resilience Plan of the Republic of Bulgaria, project N:
BG-RRP-2.004-0008-C01. The research of S.I.  is partially supported  by Contract KP-06-H72-1/05.12.2023 with the National Science Fund of Bulgaria, Contract 80-10-181 / 22.4.2024   with the Sofia University "St.Kl.Ohridski" and  the National Science Fund of Bulgaria, National Scientific Program ``VIHREN", Project KP-06-DV-7.

\section{Paraquaternionic contact manifolds and the canonical connection}

\label{s:review} In this section we will briefly review the basic notions of
paraquaternionic contact geometry and recall some results from \cite{CIZ}.

The algebra $pQ$ of para quaternions  (sometimes called  split quaternions \cite{Swann})  is a four-dimensional real vector space with
basis ${1, r_1,r_2,r_3}$, satisfying
$
r_1^2= r_2^2=1,\quad r_3^2=-1, \quad r_1r_2=-r_2r_1=r_3.
$

This carries a natural indefinite inner product given by $<p,q>= Re( \bar pq)$, where
$p = t+r_3x+r_1y+r_2z$ has $\bar p =  t-r_3x-r_1y-r_2z$. We have $||p||^2 = t^2 +x^2 -y^2 -z^2$,
so a metric of signature (2,2). This norm is multiplicative, $||pq||^2 = ||p||^2||q||^2$,
but the presence of elements of length zero means that $pQ$ contains zero divisors.


A paraquaternionic contact (pqc) manifold $(M, g,
pQ)$ is a $4n+3$ dimensional manifold $M$ with a codimension three
distribution $H$ equipped with a metric $g$ of neutral signature (2n,2n) and a $Sp(n,\mathbb R)Sp(1,\mathbb R)$ structure,
i.e., we have
\begin{enumerate}
\item[i)] a  rank-three bundle $\mathbb{PQ}$, consisting of (1,1)-tensors on $H$, locally generated
by two almost para complex structures $I_1,I_2$ and an almost complex structure $I_3$ on $H$, satisfying
the identities of the imaginary unit para quaternions,
\begin{equation}\label{paraq}I_s^2=\epsilon_s,\quad I_iI_j=-I_jI_i=-\epsilon_kI_k,  \quad where \quad\epsilon_1=\epsilon_2=-\epsilon_3=1, 
\end{equation}
which are paraquaternionic compatible with the neutral  metric $g$ on $H$,
\begin{equation}\label{param} g(I_s.,I_s.)=-\epsilon_sg(.,.).
\end{equation}
\item[ii)] $H$ is locally given as the kernel of a 1-form
$\eta=(\eta_1,\eta_2,\eta_3)$ with values in $\mathbb{R}^3, H=\cap_{s=1}^3 Ker\, \eta_s$ and the following compatibility condition holds 
\begin{equation}\label{ccon}-2\epsilon_sg(I_sX,Y)\ =\ d\eta_s(X,Y), \quad X,Y\in H.
\end{equation}
\end{enumerate}
A pqc manifold $(M, \bar g,\mathbb{PQ} )$ is called paraquaternionic contact conformal (pqc conformal) to $(M, g,\mathbb{Q} )$ if $\bar g=\nu g$, for a nowhere vanishing smooth function $\nu$. In that case, if $\bar\eta$ is a
corresponding associated one-form with paraquaternionic structure $\bar I_s$,
we have $\bar\eta\ =\ \nu\, \Psi\,\eta$ for some $\Psi\in SO(1,2)$. In particular, starting with a pqc manifold $%
(M, \eta)$ and defining $\bar\eta\ =\ \nu\, \eta$ we obtain a pqc manifold $%
(M, \bar\eta)$, pqc conformal to the original one.

On a paraquaternionic contact manifold there exists a canonical connection
defined in \cite{CIZ}.
\begin{thrm}
\cite{CIZ}\label{biqcon} {Let $(M, g, pQ)$ be a paraquaternionic contact
manifold} of dimension $4n+3>7$ with a fixed neutral metric $g$ on $H$. Then there exists a unique supplementary subspace $V$ to $H$ in $%
TM$, locally generated by a  vector fields $\{\xi_1,\xi_2,\xi_3\}$, satisfying the conditions
\begin{equation}  \label{xi}
\begin{aligned} \eta_s(\xi_t)=\delta_{st}, \qquad (\xi_s\lrcorner
d\eta_s)_{|H}=0,\\ (\xi_j\lrcorner d\eta_i)_{|H}=\epsilon_k(\xi_i\lrcorner
d\eta_j)_{|H},
\end{aligned}
\end{equation}
and a unique connection $\nabla$ with
torsion $T$ on $M^{4n+3}$, such that:
\begin{enumerate}
\item[i)] $\nabla$ preserves the splitting $H\oplus V$ and the $Sp(n,\mathbb R)Sp(1,\mathbb R)$-structure
on $H$, $\nabla g=0, \nabla pQ\subset pQ$;
\item[ii)] for $X,Y\in H$, one has $T(X,Y)=-[X,Y]_{|V}$;
\item[iii)] for $\xi\in V$, the endomorphism $T(\xi,.)_{|H}$ of $H$ lies in $%
(sp(n,\mathbb R)\oplus sp(1,\mathbb R))^{\bot}\subset gl(4n)$;
\item[iv)] the connection on $V$ is induced by the natural identification $\varphi
$ of $V$ with the subspace $sp(1,\mathbb R)$ of the endomorphisms of $H$, i.e. $%
\nabla\varphi=0$.
\end{enumerate}
\end{thrm}
We shall call the above connection \emph{the canonical pqc connection}.
The vector fields $\xi_1,\xi_2,\xi_3$ are called Reeb vector fields or
fundamental vector fields.

The fundamental 2-forms $\omega_s$
of the pqc structure $pQ$ are defined by

\centerline{$
-2\epsilon_s\omega_{s|H}\ =\ \, d\eta_{s|H},\qquad \xi\lrcorner\omega_s=0,\quad \xi\in
V.
$}

The second condition ii) yields that the torsion restricted to $H$ has the form
\begin{equation}  \label{torha}
T(X,Y)=-[X,Y]_{|V}=-2\sum_{s=1}^3\epsilon_s\f_s(X,Y)\xi_s=-2\omega_1(X,Y)\xi_1-2\omega_2(X,Y)\xi_2+2\omega_3(X,Y)\xi_3.
\end{equation}
The fourth condition iv) implies 
\be\label{conef}
\begin{split}
\nabla I_i=-\alpha_j\otimes I_k+\epsilon_k\alpha_k\otimes I_j,\quad
\nabla\xi_i=-\alpha_j\otimes\xi_k+\epsilon_k\alpha_k\otimes\xi_j,
\end{split}
\ee
where the  $sp(1,\mathbb R)$-connection 1-forms $\alpha_s$ are determined by the pqc structure \cite[Theorem~3.8]{CIZ}.

{\ If the dimension of $M$ is seven, the conditions \eqref{xi} do not
always hold. It is shown in \cite{CIZ} that if we additionally assume the
existence of Reeb vector fields as in \eqref{xi}, then Theorem~\ref{biqcon}
holds. Henceforth, by a pqc structure in dimension $7$ we shall 
mean a pqc structure satisfying \eqref{xi}.}

Notice that equations \eqref{xi} are invariant under the natural $SO(1,2)$
action. Using the triple of Reeb vector fields we extend $g$ to a metric on $%
M$ by requiring 
$span\{\xi_1,\xi_2,\xi_3\}=V\perp H \text{ and } g(\xi_s,\xi_k)=-\epsilon_s\delta_{sk}.
$ 
\hspace{2mm} \noindent The extended metric does not depend on the action of $%
SO(1,2)$ on $V$, but it changes in an obvious manner if $\eta$ is multiplied
by a conformal factor. Clearly, the canonical pqc connection preserves the
extended metric on $TM, \nabla g=0$. 
\subsection{The torsion endomorphism}
The properties of the canonical pqc connection are encoded in the properties of
the torsion endomorphism $T_{\xi}=T(\xi,.) : H\rightarrow H, \quad \xi\in V$.
Recall that any endomorphism $\Psi$ of $H$ can be decomposed with respect to the
paraquaternionic structure $(pQ,g)$ uniquely into
$Sp(n,\mathbb R)$-invariant
parts as follows 
$
\Psi=\Psi^{+++}+\Psi^{+--}+\Psi^{-+-}+\Psi^{--+}, 
$ 
where $\Psi^{+++}$ commutes with all three $I_i$,
$\Psi^{+--}$ commutes with $I_1$ and anti-commutes with the other
two and etc.  

The two $Sp(n,\mathbb R)Sp(1,\mathbb R)$-invariant components are
given by
\begin{equation}
{\label{New21}} \Psi_{[3]}=\Psi^{+++}, \qquad
\Psi_{[-1]}=\Psi^{+--}+\Psi^{-+-}+\Psi^{--+}.
\end{equation}
\noindent Denoting the corresponding (0,2) tensor via $g$ by the
same letter one sees that the $Sp(n,\mathbb R)Sp(1,\mathbb R)$-invariant components are
the projections on the eigenspaces of the Casimir operator
\begin{equation}  \label{e:cross}
\dag \ =-\ I_1\otimes I_1\ -\ I_2\otimes I_2\ +\ I_3\otimes I_3,
\end{equation}
corresponding to the eigenvalues $3$ and $-1$, respectively.
If $n=1$ then the space of symmetric endomorphisms commuting with all
$I_i, i=1,2,3$ is 1-dimensional, i.e. the [3]-component of any
symmetric
endomorphism $\Psi$ on $H$ is proportional to the identity, $\Psi_{[3]}=%
\frac{Tr(\Psi)}{4}Id_{|H}$.

Decomposing the torsion endomorphism $T_{\xi}\in(sp(n)+sp(1))^{\perp}$ into a symmetric part $\tb_{\xi}$ and an anti-symmetric part $T^a_{\xi}$,
we  define the tensors $\tau(X,Y)$ and $\mu(X,Y)$ on $H$ by 
\be\label{deftau}
\begin{split}
\tau(X,Y)=-\epsilon_i\tb(\xi_i,I_iX,Y)-\epsilon_j\tb(\xi_j,I_jX,Y)-\epsilon_k\tb(\xi_k,I_kX,Y);\\
\mu(X,Y)=\epsilon_s\ta(\xi_s,I_sX,Y).
\end{split}
\ee
The tensors $\tau$ and $\mu$ do not depend on the particular choice of the Reeb vector fields and are invariant under the natural action of $SO(1,2)$.

 We summarize the description of the torsion from \cite{CIZ} in the following Proposition. 
\begin{prop}
\cite{CIZ}\label{torb} 
The tensor $\tau$ on $H$ is symmetric, trace-free,  belongs to the [-1]- component and determines the symmetric part of the torsion endomorphism, i.e. it satisfies the relations
\begin{gather}\label{tau-trfree}
\tau(X,Y)=\tau(Y,X), \quad \tau(e_a,e_a)=\tau(I_se_a,e_a)=0;\\\label{tau-sym}
\tau(X,Y)-\tau(I_1,X,I_1,Y)-\tau(I_2X,I_2Y)+\tau(I_3X,I_3Y)=0;\\\label{tau-1}
\tb(\xi_s,X,Y)=-\frac14\Big[\tau(I_sX,Y)+\tau(X,I_sY) \Big].
\end{gather}
The tensor $\mu$ is symmetric, trace-free,  has the properties

\be\label{propmu}
\mu(I_sX,I_sY)=-\epsilon_s\mu(X,Y),
\ee
and belongs to the [3]- component and determines the skew-symmetric part of the torsion endomorphism 
\be\label{mus}
\ta(\xi_s,X,Y)=\mu(I_sX,Y).
\ee
If the dimension is seven then $\mu=0$.
\end{prop}

\subsection{The curvature and the  Ricci type tensors}

Let $R=[\nabla,\nabla]-\nabla_{[\ ,\ ]}$ be the curvature tensor of
$\nabla$ and  denote the curvature tensor of type (0,4) by the
same letter. 
The first Bianchi identity for the canonical pqc connection reads 
\begin{equation}  \label{bian1}
\sum_{(A,B,C)}\Bigl\{R(A,B,C,D)\Bigr\}= \sum_{(A,B,C)}\Bigl\{ (\nabla_AT)(B,C,D) +
T(T(A,B),C,D)\Bigr\}=b(A,B,C,D),
\end{equation}
where $\sum_{(A,B,C)}$ denotes the cyclic sum and $A,B,C,D \in \Gamma(TM)$.

The curvature of a metric connection  is skew-symmetric with respect to
the last two arguments, $R(A,B,C,D)=-R(A,B,D,C)$. It follows directly from the first Bianchi
identity \eqref{bian1} that
\begin{multline}  \label{zam}
2Zam(A,B,C,D)=2R(A,B,C,D)-2R(C,D,A,B)\\=\ b(A,B,C,D) + \ b(B,C,D,A)
- \ b(A,C,D,B)-\ b(A,B,D,C).
\end{multline}
The horizontal Ricci tensor and the scalar curvature $Scal$ of the canonical pqc 
connection, called \emph{pqc Ricci tensor} and \emph{pqc scalar 
curvature}, respectively, are defined by
\begin{equation*}  \label{e:horizontal ricci}
Ric(X,Y)=R(e_a,X,Y,e_a), 
Scal=Ric(e_a,e_a).
\end{equation*}
The curvature of the canonical pqc connection admits  several 
traces, defined in \cite{CIZ} by
\begin{gather*}
4n\rho_s(A,B)=R(A,B,e_a,I_se_a), \hspace{1mm} 4n\varrho_s(A,B)=R(e_a,I_se_a,A,B), \hspace{1mm}
4n\zeta_s(A,B)=R(e_a,A,B,I_se_a).
\end{gather*}
The curvature operator $R(B,C)$ preserves the pqc structure on $M$  since the connection $\C$ preserves it. In particular, $R(B,C)$ preserves the splitting $H\oplus V$ and the paraquaternionic structure on $H$, so $R(B,C)\in sp(n,\mathbb R)\oplus sp(1,\mathbb R)$ on $H$. We use the following result established in \cite{CIZ}.
\begin{lemma}\cite{CIZ}
On a pqc manifold the next identities hold
\begin{gather}\label{rjr}
\epsilon_iR(A,B,I_iX,I_iY)+R(A,B,X,Y)=-2\epsilon_j\rho_j(A,B)\f_j(X,Y)-2\epsilon_k\rho_k(A,B)\f_k(X,Y);\\
\label{rhov}
 R(A,B)\xi_i=-2\epsilon_i\rho_k(A,B)\xi_j+2\epsilon_i\rho_j(A,B)\xi_k, \qquad
 \rho_i=\frac12\Big[\epsilon_kd\alpha_i-\epsilon_j\alpha_j\wedge\alpha_k \Big];
\end{gather}
\end{lemma}



It is shown in \cite{CIZ} that all horizontal Ricci type
contractions of the curvature of the canonical pqc connection can be
expressed in terms of the torsion endomorphism. 
We utilize the following from \cite{CIZ}.

\begin{thrm}
\cite{CIZ}\label{sixtyseven} On a $(4n+3)$-dimensional pqc manifold the horizontal Ricci
tensors $Ric$ and $\zeta_s(X,I_sY)$ are symmetric,  the
horizontal Ricci tensors $\rho_s(X,I_sY), \varrho_s(X,I_sY)$ are
symmetric (1,1) tensors with respect to $I_s$ 
and the next
formulas hold

\begin{gather}\label{ricci}
Ric(X,Y)  = \frac{Scal}{4n}g(X,Y)+(2n+2)\tau(X,Y)
+(4n+10)\mu(X,Y);\\\label{ricciformf}
 \rho_s(X,I_sY)  = \epsilon_s\frac{Scal}{8n(n+2)}g(X,Y)
+\frac12\Bigl[\epsilon_s\tau(X,Y)-\tau(I_sX,I_sY)\Bigr]+2\epsilon_s\mu(X,Y);\\
\label{riccitau}
 \varrho_s(X,I_sY)   = \epsilon_s\frac{Scal}{8n(n+2)}g(X,Y)
+\frac{n+2}{2n}\Bigl[\epsilon_s\tau(X,Y)-\tau(I_sX,I_sY)\Bigr];\\
\label{riccizeta}
\epsilon_s \zeta_s(X,I_sY) \  =\ -\frac{Scal}{16n(n+2)}g(X,Y)-
\frac{2n+1}{4n}\tau(X,Y)+\epsilon_s\frac{1}{4n}\tau(I_sX,I_sY)-\frac{2n+1}{2n}\mu(X,Y);\\
\label{torv}
 T(\xi_{i},\xi_{j}) =\epsilon_k\frac
{Scal}{8n(n+2)}\xi_{k}-[\xi_{i},\xi_{j}]_{H};
\\\label{vertor}
 T(\xi_i,\xi_j,I_kX)
=\rho_k(I_jX,\xi_i)=-\rho_k(I_iX,\xi_j)=\omega_k([\xi_i,\xi_j],X);
\\\label{ricvert1}
 -\epsilon_i\rho_i(\xi_i,\xi_j)-\epsilon_k\rho_k(\xi_k,\xi_j)=\frac{1}{16n(n+2)}\xi_j(Scal).
\end{gather}
For $n=1$ the above formulas hold with $\mu=0$.
\end{thrm}
Clearly, the condition $Ric=\frac{Scal}{4n} g$ is equivalent to $\tau=\mu=0$. It was shown in \cite{CIZ} that the torsion endomorphism of a pqc manifold vanishes exactly when it is locally para 3-Sasakian 
provided the pqc scalar curvature is nonvanishing  and the dimension is bigger than seven.

It is  established in \cite{CIZ} that the whole curvature is determined from the horizontal curvature. We will use  the next result proved in  \cite{CIZ}.
\begin{thrm}\cite{CIZ}
\label{bianrrre} On a pqc manifold the curvature of the canonical connection satisfies the
equalities:
\begin{multline}\label{vert1}
 R(\xi_i,X,Y,Z)= -(\nabla_X\mu)(I_iY,Z)
 -\frac14\Big[(\nabla_Y\tau)(I_iZ,X)+(\nabla_Y\tau)(Z,I_iX)\Big]\\ +\frac14\Big[(\nabla_Z\tau)(I_iY,X)+
 (\nabla_Z\tau)(Y,I_iX)\Big]
 +\omega_j(X,Y)\rho_k(I_iZ,\xi_i)-\omega_k(X,Y)\rho_j(I_iZ,\xi_i)\\
-\omega_j(X,Z)\rho_k(I_iY,\xi_i)+\omega_k(X,Z)\rho_j(I_iY,\xi_i)
-\omega_j(Y,Z)\rho_k(I_iX,\xi_i)+\omega_k(Y,Z)\rho_j(I_iX,\xi_i).
\end{multline}
\begin{multline}\label{vert2}
 R(\xi_i,\xi_j,X,Y)=(\nabla_{\xi_i}\mu)(I_jX,Y)-(\nabla_{\xi_j}\mu)(I_iX,Y) +\epsilon_j(\nabla_X\rho_k)(I_iY,\xi_i)\\
 -\frac14\Big[(\nabla_{\xi_i}\tau)(I_jX,Y)+(\nabla_{\xi_i}\tau)(X,I_jY)\Big]
 +\frac14\Big[(\nabla_{\xi_j}\tau)(I_iX,Y)+(\nabla_{\xi_j}\tau)(X,I_iY)\Big]\\ +\epsilon_k\frac{Scal}{8n(n+2)}T(\xi_k,X,Y)
 -T(\xi_j,X,e_a)T(\xi_i,e_a,Y)+T(\xi_j,e_a,Y)T(\xi_i,X,e_a),
\end{multline}
where the Ricci 2-forms are given by

\begin{eqnarray}\label{vert023}
3(2n+1)\rho_i(\xi_i,X)=-\epsilon_i\frac14(\nabla_{e_a}\tau)(e_a,X)-\frac34(\nabla_{e_a}\tau)(I_ie_a,I_iX)
+\epsilon_i(\nabla_{e_a}\mu)(X,e_a)\\\nonumber
-\epsilon_i\frac{2n+1}{16n(n+2)}X(Scal);\\
\label{vert024}
3(2n+1)\rho_i(I_kX,\xi_j)=-3(2n+1)\rho_i(I_jX,\xi_k)=-\frac{(2n+1)(2n-1)}{16n(n+2)}X(Scal)\\\nonumber
+2(n+1)(\nabla_{e_a}\mu)(X,e_a)
+\frac{4n+1}4(\nabla_{e_a}\tau)(e_a,X)-\epsilon_i\frac34(\nabla_{e_a}\tau)(I_ie_a,I_iX);
\\ \label{div}
(n-1)(\nabla_{e_a}\tau)(e_a,X)+2(n+2)(\nabla_{e_a}\mu)(e_a,X)-\frac{(n-1)(2n+1)}{8n(n+2)}d(Scal)(X)=0.
\end{eqnarray}
\end{thrm}
As a consequence of Theorem~\ref{bianrrre}  one gets the next result originally proved in \cite{CIZ}.
\begin{prop}\cite{CIZ}
\label{hflat} A pqc manifold is locally isomorphic to the
paraquaternionic Heisenberg group exactly when the curvature of the
canonical connection restricted to $H$ vanishes, $R_{|_H}=0$.
\end{prop}
\section{paraquaternionic Heisenberg group and the paraquaternionic Cayley
transform}\label{s:standard structures} Since our goal is to classify
paraquaternionic contact manifolds locally conformal to the  paraquaternionic
Heisenberg group, we recall its definition from \cite{CIZ}, define  the paraquaternionic Cayley transform and show that it is a local paraquaternionic contact authomrphism between the para 3-Sasakian structure on the pseudo-sphere and the paraquaternionic Heisenberg group.

\subsection{The paraquaternionic Heisenberg group}  As a manifold the paraquaternionic Heisenberg group
of topological dimension $4n+3$ is   $G(pH) = pH^n\times Im(pH)$ with
the group law given by $$(q', \omega') = (q_o, \omega_o)_o(q, \omega) = (q_o + q, \omega_o + \omega + 2 Im(q_o\bar q)),$$
where $q,q_o\in pH^n$ and $\omega,\omega_o\in Im(pH)$.

On $G(pH)$ we define the paraquaternionic contact form in paraquaternionic variables as follows
$$\tilde\Theta=(\tilde\Theta_3,\tilde\Theta_1\tilde\Theta_2)=\frac12(d\omega-qd\bar q+dq\bar q).$$
In real coordinates we get
\begin{equation}\label{pqh}
\begin{split}
\tilde\Theta_3=\frac12dx-x^adt^a+t^adx^a-z^ady^a+y^adz^a;\\
\tilde\Theta_1=\frac12dy-y^adt^a-z^adx^a+t^ady^a+x^adz^a;\\
\tilde\Theta_2=\frac12dz-z^adt^a+y^adx^a-x^ady^a+t^adz^a.
\end{split}
\end{equation}
The structure equations of $G(pH)$ are
\begin{equation*}
\begin{split}
d\tilde\Theta_3=2(dt^a\wedge dx^a+dy^a\wedge dz^a);\\
d\tilde\Theta_1=2(dt^a\wedge dy^a+dx^a\wedge dz^a);\\
d\tilde\Theta_2=2(dt^a\wedge dz^a-dx^a\wedge dy^a).
\end{split}
\end{equation*}
The left-invariant horizontal  vector fields $T_a, X_a=J_3T_a, Y_a=-J_1T_a, Z_a=-J_2T_a$ are given by
\begin{equation*}
\begin{split}
T_a=\dta  + 2x^a\dx +2y^a\dy+2z^a\dz;\qquad X_a=\dxa  + 2t^a\dx-2z^a\dy+2y^a\dz;\\
Y_a=\dya  + 2z^a\dx-2t^a\dy-2x^a\dz;\qquad Z_a=\dza  -2y^a\dx+2x^a\dy-2t^a\dz.
\end{split}
\end{equation*}
The horizontal metric of signature (2n,2n) is defined by
$$g(T_a,T_a)=g(X_a,X_a)=-g(Y_a,Y_a)=-g(Z_a,Z_a)=1.$$
The central  (left-invariant vertical) Reeb vector  fields are
$\xi_3=2\dx, \quad \xi_1=2\dy, \quad \xi_2=2\dz
$
and a straightforward calculation shows the following commutation  relations
$$[J_iT_a,T_a]=-2\epsilon_i\xi_i,\qquad [J_iT_a,J_jT_a]=2\epsilon_k\xi_k.
$$
It is easy to verify that the left-invariant flat connection on $G(pH)$ coincides with the canonical pqc connection of the pqc manifold $(G(pH),\tilde\Theta)$.
\subsection{An embeding of the paraquaternionic Heisenberg group $G(pH)$}
Consider the hypersurface 
$$\Sigma\subset pH^n\times pH:\Sigma=(q',p')\in pH\times pH:Re(p')=-|q'|^2.$$
Clearly, $\Sigma$ is the $0$-level set of $\rho=|q'|^2+t$ and 
\be\label{ppo}d\rho=q'd\bar{q'}+dq'\bar{q'}+dt=2(t^adt^a+x^adx^a-y^ady^a-z^adz^a)+dt.
\ee
The standard paraquaternionic structure $J_3,J_1,J_2$ on $\mathbb R^{4n+4}$, induced by the multiplication on the right by the para quaternions $r_3,r_1,r_2\in pH^{n+1}$, is

\begin{equation}\label{rrr}
\begin{split}
J_3dt^a=dx^a,\quad J_1dt^a=dy^a, \quad J_2dt^a=dz^a,\\
J_3dy^a=-dz^a,\quad J_1dx^a=dz^a,\quad J_2dx^a=-dy^a,\\
J_3dt=dx,\quad J_1dt=dy, \quad J_2dt=dz.
\end{split}
\end{equation}
Combining \eqref{ppo} with \eqref{rrr} and comparing with \eqref{pqh} we get
\begin{equation*}
\begin{split}
J_3d\rho=2(t^adx^a-x^adt^a+y^adz^a-z^ady^a)+dx=2\tilde\Theta_3;\\
J_1d\rho=2(t^ady^a+x^adz^a-y^adt^a-z^adx^a)+dy=2\tilde\Theta_1;\\
J_2d\rho=2(t^adz^a-x^ady^a+y^adx^a-z^adt^a)+dz=2\tilde\Theta_2.
\end{split}
\end{equation*}
We identify $G(pH)$ with $\Sigma$ by $(q',\omega')\rightarrow (q',p'=-|q'|^2+\omega')$. Since $dp'=-q'd\bar{q'}-dq'\bar{q'}+d\omega'$, we  write
\[\tilde\Theta=\frac12(d\omega-q'd\bar{q'}+dq'\bar{q'})=\frac12dp'+dq'\bar{q'}.\]
Taking into account that $\tilde\Theta$ is pure imaginary we can  write the last equation in the form
\be\label{phsp}\tilde\Theta=\frac14(dp'-d\bar{p'})+\frac12(dq'\bar{q'}-q'd\bar{q'}).
\ee
\subsection{The para 3-Sasakian pseudo sphere and the paraquaternionic Cayley transform} The second explicit example is the  pqc-structure on the para 3-Sasakian pseudo-sphere. The para 3-Sasakian structure on the pseudo-sphere (hyperboloid) $pS^{4n+3}=\{|q|^2+|p|^2=1\}\subset pH^n\times pH$ is inherited from the standard flat hypersymplectic structure on $\mathbb R^{4n+4}= pH^n\times pH$.  In paraquaternionic variables,  the  pqc 1-form on the pseudo sphere $pS^{4n+3}=\{|q|^2+|p|^2=1\}\subset pH^n\times pH$ is defined as follows
\begin{equation}\label{p3sas1}
\tilde\eta=dq.\bar q+dp.\bar p-q.d\bar q-p.d\bar p.
\end{equation}
We consider  the map from the pseudo-sphere $pS^{4n+3}$ minus the points $\Sigma_0$,  
$$\Sigma_0=(q,p)\in pS^{4n+3}: |p-1|^2=(t-1)^2+x^2-y^2-z^2=0$$
to the paraquaternionic Heisenberg group $G(pH)\cong\Sigma$, defined by 
\[\mathbb C:\Big(pS^{4n+3}-\Sigma_0\Big)\rightarrow\Sigma,\quad  (q',p')=\mathbb C\Big((q,p)\Big), \quad q'=(p-1)^{-1}q,\quad p'=(p-1)^{-1}(p+1)
\]
since 
$Re(p')=Re\Big(\frac{ (\bar p-1)(p+1)}{|p-1|^2}\Big)=\frac{|p|^2-1}{|p-1|^2}=-\frac{|q|^2}{|p-1|^2}=-|q'|^2.
$

The inverse map  $(q,p)=\mathbb C^{-1}\Big((q',p')\Big)$ is given by
\[q=2(p'-1)^{-1}q',\quad  p=(p'-1)^{-1}(p'+1).
\]
We call this map \emph{paraquaternionic Cayley transform}.

An easy calculation gives
\begin{equation}\label{cay}
\begin{split}
dp'=-2(p-1)^{-1}.dp.(p-1)^{-1};\quad 
dq'=(p-1)^{-1}.\Big[dq-dp.(p-1)^{-1}.q  \Big].
\end{split}
\end{equation}
Using \eqref{phsp} together with \eqref{cay}, we calculate

\be\label{cal}
\begin{split}
2\mathbb C^*\tilde\Theta=-(p-1)^{-1}.dp.(p-1)^{-1}+(\bar p-1)^{-1}.d\bar p.(\bar p-1)^{-1}\\
+(p-1)^{-1}.\Big[dq-dp.(p-1)^{-1}.q\Big].\bar q.(\bar p-1)^{-1}\quad
- (p-1)^{-1}.q.\Big[d\bar q-\bar q.(\bar p-1)^{-1}.d\bar p\Big].(\bar p-1)^{-1}\\
=(p-1)^{-1}.\Big[dq.\bar q-q.d\bar q\Big].(\bar p-1)^{-1}\\
- (p-1)^{-1}.\Big[dp.(p-1)^{-1}(\bar p-1)+|q|^2dp.(p-1)^{-1}].(\bar p-1)^{-1}\\
+ (p-1)^{-1}.\Big[(p-1)(\bar p-1)^{-1}.d\bar p+|q|^2.(\bar p-1)^{-1}.d\bar p].(\bar p-1)^{-1}\\
=(p-1)^{-1}.\Big[dq.\bar q-q.d\bar q\Big].(\bar p-1)^{-1}
-(p-1)^{-1}.\Big[dp.(p-1)^{-1}(\bar p-\bar p.p)].(\bar p-1)^{-1}\\
+(p-1)^{-1}.\Big[(p-p\bar p)(\bar p-1)^{-1}.d\bar p].(\bar p-1)^{-1}\\
=(p-1)^{-1}.\Big[dq.\bar q-q.d\bar q+dp.\bar p-p.d\bar p\Big].(\bar p-1)^{-1}
=\frac1{|p-1|^2}\lambda.\tilde\eta.\bar{\lambda},
\end{split}
\ee
where $\lambda=\frac{|p-1|}{p-1}$ is a unit paraquaternion and $\tilde\eta$ is the standard paraquaternionic contact form on the pseudo-sphere $pS^{4n+3}$, given by \eqref{p3sas1}.

Since $p-1=2(p'-1)^{-1}$ we  have $\lambda=\frac1{|p'-1|}(p'-1)$ and we can put \eqref{cal} into the form
\be\label{cal-}
\lambda.(\mathbb C^*)^{-1}.\bar{\lambda}=\frac8{|p'-1|^2}\tilde\Theta.
\ee
Thus, we prove the next
\begin{prop}\label{phps}
The paraquaternionic Heisenberg group $G(pH)$ and the para 3-Sasakian pseudo-sphere $pS^{4n+3}$ are locally pqc conformally equivalent via the paraquaternionic Cayley transform
\end{prop}

\section{paraquaternionic contact conformal curvature. Proof of Theorem~\ref{main1}}
In this section we define the paraquaternionic contact conformal curvature and prove Theorem~\ref{main1}.
\subsection{Conformal transformations}
\label{s:conf transf}

A conformal paraquaternionic contact transformation between two
paraquaternionic contact manifold is a diffeomorphism $\Phi$, which
satisfies $\Phi^*\eta=\mu\ \Psi\cdot\eta$ for some positive smooth
function $\mu$ and some matrix $\Psi\in SO(1,2)$ with smooth functions
as entries, where $\eta=(\eta_1,\eta_2,\eta_3)^t$ is considered as
an element of $\mathbb{R}^3$. The canonical pqc connection introduced in \cite{CIZ} does not
change under the action of $SO(1,2)$, i.e., the canonical connection of $\Psi\cdot\eta$ and $%
\eta$ coincides. Hence, as we study pqc conformal transformations we may consider
only transformations $\Phi^*\eta\ =\ \mu\ \eta$.

\subsection{Paraquaternionic conformal transformations}
 Let $h$ be a positive smooth function on a pqc
manifold $(M, \eta)$. Let $\bar\eta=\frac{1}{2h}\eta$ be a conformal
deformation of the pqc structure $\eta$. We will denote the objects related
to $\bar\eta$ by over-lining the same object corresponding to $\eta$. Thus,
$$d\bar\eta=-(2h^2)^{-1}dh\wedge\eta\ +(2h)^{-1}d\eta, \qquad  \bar
g=(2h)^{-1}g.$$

The new triple $\{\bar\xi_1,\bar\xi_2,\bar\xi_3\}$, determined by
the conditions \eqref{xi} defining the Reeb vector fields, is
$$\bar\xi_s\ =\ 2h\,\xi_s\ +\ I_s\nabla h,$$  where $\nabla h$ is the
horizontal gradient defined by $\nabla h=dh(e_a)e_a, \quad g(\nabla h,X)=dh(X).$

The horizontal hyperbolic sub-Laplacian and the norm of the horizontal gradient are
defined by $\triangle_h h\ =\ tr^g_H(\nabla^2h)\ = g(e_a,e_a)\
\nabla^2h(e_a,e_a )$, $|\nabla h|^2\ =\
g(e_a,e_a)dh(e_a)\,dh(e_a)$, respectively. The canonical pqc connections $\nabla$ and
$\bar\nabla$ are connected by a (1,2)-tensor S,
\begin{equation}  \label{qcw2}
\bar\nabla_AB=\nabla_AB+S_AB, \qquad A,B\in\Gamma(TM).
\end{equation}
Condition \eqref{torha} yields
$
g(S_XY,Z)-g(S_YX,Z)=h^{-1}\sum_{s=1}^3\epsilon_s\omega_s(X,Y)dh(I_sZ),
$
while $\bar\nabla\bar g=0$ implies
$g(S_XY,Z)+g(S_XZ,Y)=-h^{-1}dh(X)g(Y,Z)$.
The last two equations determine $g(S_XY,Z)$,  
\begin{multline}  \label{Ivan4}
g(S_XY,Z)=-(2h)^{-1}\{dh(X)g(Y,Z)+\sum_{s=1}^3\epsilon_sdh(I_sX)\omega_s(Y,Z)\\+
dh(Y)g(Z,X)-\sum_{s=1}^3\epsilon_sdh(I_sY)\omega_s(Z,X)-dh(Z)g(X,Y)-%
\sum_{s=1}^3\epsilon_sdh(I_sZ)\omega_s(X,Y)\}.
\end{multline}
Using Theorem~\ref{biqcon} and after some calculations, we obtain
\begin{multline}  \label{New20}
g(\bar T_{\bar\xi_i}X,Y)-2hg(T_{\xi_i}X,Y)-g(S_{\bar\xi_i}X,Y)\\
=-\nabla dh(X,I_iY)-\epsilon_ih^{-1}[dh(I_kX)dh(I_jY)-dh(I_jX)dh(I_kY)].
\end{multline}
The identity $d^2=0$ yields 
$\nabla^2h(X,Y)-\nabla^2h(Y,X)=-dh(T(X,Y)).$ 
Applying \eqref{torha}, we have
\begin{equation}  \label{symdh}
\nabla^2h(X,Y)=[\nabla^2h]_{[sym]}(X,Y)+\sum_{s=1}^3 \epsilon_sdh(\xi_s)\omega_s(X,Y),
\end{equation}
where $[.]_{[sym]}$ denotes the symmetric part of the corresponding
(0,2)-tensor.  

Decompose \eqref{New20} into [3] and [-1] parts
according to \eqref{New21}, use the properties of the torsion
tensor $T_{\xi_i}$ and \eqref{deftau} to come to the next
transformation formulas 
\begin{prop}
Let $\bar\eta=\frac{1}{2h}\eta$ be a pqc conformal transformation of a given pqc structure $\eta$. Then the two parts of the torsion endomorphism transform as follows
\begin{equation}\label{defor}
\begin{split}
\bar\tau=\tau+h^{-1}[\nabla^2h]_{[sym][-1]},\\
\bar\mu=\mu+(2h)^{-1}[\nabla^2h-2h^{-1}dh\otimes dh]_{[3][0]},
\end{split}
\end{equation}
where $[\nabla^2h]_{[sym][-1]}$ denotes the (-1)-component of the symmetric part of the horizontal Hessian and $[\nabla^2h-2h^{-1}dh\otimes dh]_{[3][0]}$ is the trace-free part of the (3)-component. Explicitly,
\begin{equation}\label{pdefor}
\begin{split}
[\nabla^2h]_{[sym][-1]}(X,Y)=\frac14\Big[3\C^2h(X,Y)+\sum_{s=1}^3\epsilon_s\C^2h(I_sX,I_sY)-4\sum_{s=1}^3\epsilon_sdh(\xi_s)\f_s(X,Y)  \Big];\\
[\nabla^2h-2h^{-1}dh\otimes dh]_{[3][0]}(X,Y)=\frac14\Big[\C^2h(X,Y)-\sum_{s=1}^3\epsilon_s\C^2h(I_sX,I_sY)\\-\frac2{h}\Big(dh(X)dh(Y)-\sum_{s=1}^3\epsilon_sdh(I_sX)dh(I_sY)\Big)\Big]-\frac1{4n}\Big(\triangle_hh-\frac2{h}|\C h|^2 \Big)g(X,Y).
\end{split}
\end{equation}
The tensor $g(S_{\bar\xi_i}X,Y) $ is given by 
\begin{gather}\label{qcw3}
g(S_{\bar\xi_i}X,Y) \ =\ \frac{1}{4}\Big[\nabla^2h(X,I_iY)-\nabla^2h(I_iX,Y) -\epsilon_i\nabla^2h(I_jX,I_kY)+\epsilon_i\nabla^2h(I_kX,I_jY)\Bigr] \\\nonumber
\hskip.8truein +\ (2h)^{-1}\Bigl[\epsilon_idh(I_kX)dh(I_jY)-\epsilon_idh(I_jX)dh(I_kY)-
dh(I_iX)dh(Y)+dh(X)dh(I_iY)\Bigr] \\\nonumber
\hskip.11truein \ + \ \frac{1}{4n}\left(-\triangle h+2h^{-1}|\nabla
h|^2\right)\omega_i(X,Y) +\epsilon_idh(\xi_k)\omega_j(X,Y)-\epsilon_idh(\xi_j)\omega_k(X,Y).
\end{gather}
\end{prop}

\subsection{paraquaternionic contact conformal curvature}

Let $(M,g,\mathbb{pQ})$ be a (4n+3)-dimensional pqc manifold. We consider the
symmetric (0,2) tensor $L$ defined on $H$ by the equality
\begin{multline}  \label{lll}
L=\Bigl(\frac{1}{4(n+1)}Ric_{[-1]}+\frac{1}{2(2n+5)}Ric_{[3][0]}+ \frac{%
1}{32n(n+2)}Scal\,g\Bigr)
=\frac12\tau +\mu+\frac{Scal}{32n(n+2)}\,g,
\end{multline}
where $Ric_{[-1]}$ is the [-1]-part of the horizontal Ricci tensor, $Ric_{[3][0]}$ is the trace-free [3]-part of $Ric$ and we use the identities in Theorem~\ref{sixtyseven} to obtain the second equality.

Let us denote the trace-free part of $L$ with $L_0$, hence,
\begin{equation}  \label{l0}
L_0=\frac{1}{4(n+1)}Ric_{[-1]}+\frac{1}{2(2n+5)}Ric_{[3][0]}=\frac12\tau+\mu.
\end{equation}

The Kulkarni-Nomizu product of two (not
necessarily symmetric)  tensors is defined by
\begin{multline*}
(A\owedge B)(X,Y,Z,V):=A(X,Z)B(Y,V)+
A(Y,V)B(X,Z)-A(Y,Z)B(X,V)-A(X,V)B(Y,Z).
\end{multline*}
 We  note explicitly the following usual conventions 
$
I_s L\, (X,Y) = g(I_s L X,Y) = -L (X,I_s Y).
$

Now, define the (0,4) tensor $PWR$ on $H$ as follows
\begin{multline}\label{qcwdef}
PWR(X,Y,Z,V)=R(X,Y,Z,V)+(g\owedge L)(X,Y,Z,V)-\sum_{s=1}^3\epsilon_s(\omega _{s}\owedge
I_{s}L)(X,Y,Z,V)\\
+\frac{1}{2}\sum_{(i,j,k)}\epsilon_i\omega_i(X,Y)\Bigl[L(Z,I_iV)-L(I_iZ,V)-\epsilon_iL(I_jZ,I_kV)+\epsilon_iL(I_kZ,I_jV) %
\Bigr] \\
+\sum_{s=1}^3\epsilon_s\omega_s(Z,V)\Bigl[L(X,I_sY)-L(I_sX,Y)\Bigr]
-\frac{1}{2n}(tr L)\sum_{s=1}^3\epsilon_s\omega_s(X,Y)\omega_s(Z,V),
\end{multline}
where $\sum_{(i,j,k)}$ denotes the cyclic sum.

A substitution of \eqref{lll} and \eqref{l0} in \eqref{qcwdef},
invoking Proposition~\ref{torb} gives
\begin{multline}  \label{qcwdef1}
PWR(X,Y,Z,V)= R(X,Y,Z,V)+  (g\owedge L_0)(X,Y,Z,V)-\sum_{s=1}^3\epsilon_s(\omega_s\owedge
I_sL_0)(X,Y,Z,V)\\
+\frac12\sum_{s=1}^3\epsilon_s\Bigl[\omega_s(X,Y)\Bigl\{\tau(Z,I_sV)-\tau(I_sZ,V)\Bigr\} %
+ \omega_s(Z,V)\Bigl\{\tau(X,I_sY)-\tau(I_sX,Y)+4\mu(X,I_sY)\Bigr\}\Bigr]\\
+\frac{Scal}{32n(n+2)}\Big[(g\owedge
g)(X,Y,Z,V)-\sum_{s=1}^3\epsilon_s\Bigl((\omega_s\owedge\omega_s)(X,Y,Z,V)
+4\omega_s(X,Y)\omega_s(Z,V)\Bigr) \Big].
\end{multline}
\begin{prop}
\label{trfree} The tensor $PWR$ is completely trace-free, i.e.
\begin{equation*}
Ric(PWR)=\rho_s(PWR)=\varrho_s(PWR)=\zeta_s(PWR)=0.
\end{equation*}
\end{prop}
\begin{proof}
Proposition~\ref{torb}  and
\eqref{lll} imply  the following identities
\begin{align}\label{t01}
\tau(X,Y) =\frac12\Big[3L(X,Y)+L(I_1X,I_1Y)+L(I_2X,I_2Y)-L(I_3X,I_3Y)\Big];\hspace{2.4cm}
\\\label{u01}
\mu(X,Y)  =\frac14\Big[L(X,Y)-L(I_1X,I_1Y)
-L(I_2X,I_2Y)+L(I_3X,I_3Y)-\frac1n tr\,L\,g(X,Y)\Big];
\\\label{t1}
T(\xi_{i},X,Y)=
-\frac{1}{2}\Big[L(I_{i}X,Y)+L(X,I_{i}Y)\Big]+\mu(I_{i}X,Y)
\hspace{4.5cm}\\\notag
 =-\frac{1}{4}L(I_{i}X,Y)-\frac{3}{4}L(X,I_{i}Y)+\frac{1}{4}\epsilon_i
L(I_{k}X,I_{j}Y)-\frac{1}{4}\epsilon_iL(I_{j}X,I_{k}Y)-\frac{1}{4n}(tr\,L)\,\f_i(X,Y).
\end{align}
After substituting \eqref{t01} and \eqref{u01} into the first
four equations of Theorem~\ref{sixtyseven}, we derive

\begin{equation}\label{ricis}
\begin{aligned}
Ric(X,Y)=\frac{2n+3}{2n}tr\,L\,g(X,Y)\hspace{7.5cm}\\+\frac{8n+11}{2}L(X,Y)- \frac32\Bigl[
\epsilon_iL(I_iX,I_iY)+\epsilon_jL(I_jX,I_jY)+\epsilon_kL(I_kX,I_kY)\Bigr];\\
\rho_i(X,Y)=L(X,I_iY)-L(I_iX,Y)-\frac1{2n}tr L\,\omega_i(X,Y); \hspace{4cm}\\
\varrho_i(X,Y)=- \frac1n
tr\,L\,\omega_i(X,Y)\hspace{8cm}\\-\frac{n+2}{2n}\Bigl[%
L(I_iX,Y)-L(X,I_iY)-\epsilon_iL(I_kX,I_jY)+\epsilon_iL(I_jX,I_kY)\Bigr]; \\ 
\zeta_i(X,Y)=\frac{2n-1}{8n^2}tr\,L\,\omega_i(X,Y)\hspace{7.5cm}
\\+\frac3{8n}L(I_iX,Y)-\frac{8n+3}{8n}L(X,I_iY)- \frac1{8n}\epsilon_i\Bigl[%
L(I_kX,I_jY)-L(I_jX,I_kY)\Bigr].  
\end{aligned}
\end{equation}
Take the corresponding traces in \eqref{qcwdef} and use
\eqref{ricis} to  verify the claim.
\end{proof}
We outline the following
\begin{thrm}
\label{bianrrr} On a pQC manifold the curvature of the canonical connection satisfies the
equalities:
\begin{multline}\label{zamiana}
 R(X,Y,Z,V)-R(Z,V,X,Y)=-2\sum_{s=1}^3\epsilon_s\Big[\omega_s(X,Y)\mu(I_sZ,V)-
\omega_s(Z,V)\mu(I_sX,Y)\Big]\\+\frac12\sum_{s=1}^3\epsilon_s\Big[\omega_s(Y,Z)\Big(\tau(I_sX,V)+\tau(X,I_sV)\Big)+\omega_s(X,V)\Big(\tau(I_sZ,Y)+\tau(Z,I_sY)\Big) \Big]\\
 -\frac12\sum_{s=1}^3\epsilon_s\Big[\omega_s(X,Z)\Big(\tau(I_sY,V)+\tau(Y,I_sV)\Big)+
\omega_s(Y,V)\Big(\tau(I_sZ,X)+\tau(Z,I_sX)\Big)\Big].
\end{multline}
The [3]-componenet of the horizontal curvature with respect to the first two arguments is given by
\begin{multline}\label{comp1}
 3R(X,Y,Z,V)+\sum_{s=1}^3\epsilon_sR(I_sX,I_sY,Z,V)\\
= 2\Big[g(Y,Z)\tau(X,V)+g(X,V)\tau(Z,Y)\Big]-2\Big[g(Z,X)\tau(Y,V)+ g(V,Y)\tau(Z,X)\Big]
\\
+2\sum_{s=1}^3\epsilon_s\Big[\omega_s(Y,Z)\tau(X,I_sV)+\omega_s(X,V)\tau(Y,I_sZ)\Big]
-2\sum_{s=1}^3\epsilon_s\Big[\omega_s(X,Z)\tau(Y,I_sV)+\omega_s(Y,V)\tau(X,I_sZ)\Big]
\\
-2\sum_{s=1}^3\epsilon_s\Big[\omega_s(X,Y)\Big(\tau(Z,I_sV)-\tau(I_sZ,V)\Big)-4\omega_s(Z,V)\mu(I_sX,Y)\Big] +\frac{Scal}{2n(n+2)}\sum_{s=1}^3\epsilon_s
\omega_s(X,Y)\omega_s(Z,V).
\end{multline}
\end{thrm}

\begin{proof}
The first Bianchi identity \eqref{bian1}, its consequence \eqref{zam} together with
the help of Proposition~\ref{torb} and \eqref{torha} imply
the identity \eqref{zamiana} in a straightforward way.


Set $A=Z, B=V$ in \eqref{rjr} 
and take a cyclic sum to get 
\begin{multline}\label{rjr1}
\sum_{s=1}^3\epsilon_sR(Z,V,I_sX,I_sY)+3R(Z,V,X,Y)=-4\sum_{s=1}^3\epsilon_s\rho_j(Z,V)\f_s(X,Y)\\=
\sum_{s=1}^3\epsilon_s\Big[ \frac{Scal}{2n(n+2)}\f_s(Z,V)+2\tau(I_sZ,V)-2\tau(Z,I_sV)-8\mu(Z,I_sV)\Big]\f_s(X,Y),
\end{multline}
where we used \eqref{ricciformf} to obtain the last equality. Furthermore, 
we  apply \eqref{rjr1} to get
\begin{multline}\label{comp2}
3R(X,Y,Z,V)+\sum_{s=1}^3\epsilon_sR(I_sX,I_sY,Z,V)=3R(Z,V,X,Y)+3Zam(X,Y,Z,V)\\
+\sum_{s=1}^3\epsilon_s\Big[R(Z,V,I_sX,I_sY)+Zam(I_sX,I_sY,Z,V)\Big]=3R(Z,V,X,Y)+3Zam(X,Y,Z,V)-3R(Z,V,X,Y)\\
+\sum_{s=1}^3\epsilon_s\Big[ \frac{Scal}{2n(n+2)}\f_s(Z,V)+2\tau(I_sZ,V)-2\tau(Z,I_sV)-8\mu(Z,I_sV)\Big]\f_s(X,Y)+
\sum_{s=1}^3\epsilon_sZam(I_sX,I_sY,Z,V)\\=3Zam(X,Y,Z,V)+\sum_{s=1}^3\epsilon_sZam(I_sX,I_sY,Z,V)\\
+\sum_{s=1}^3\epsilon_s\Big[ \frac{Scal}{2n(n+2)}\f_s(Z,V)+2\tau(I_sZ,V)-2\tau(Z,I_sV)-8\mu(Z,I_sV)\Big]\f_s(X,Y).
\end{multline}
Next, after substituting \eqref{zamiana} into \eqref{comp2} we obtain \eqref{comp1} by series of straightforward calculations.
\end{proof}
Comparing \eqref{qcwdef1} with \eqref{comp1} we  obtain the next
\begin{prop}\label{main0} On a pqc manifold
the [-1]-part with respect to the first two arguments of the tensor $PWR$ vanishes identically,
$$PWR_{[-1]}(X,Y,Z,V)=\frac14\Big[3PWR(X,Y,Z,V)+\sum_{s=1}^3\epsilon_sPWR(I_sX,I_sY,Z,V)\Big]=0.$$
The [3]-component with respect to the first two arguments of the tensor
$PWR$ is determined completely by the torsion and the pqc scalar curvature as follows
\begin{multline}\label{qccm}
PWR_{[3]}(X,Y,Z,V)=\frac14\Big[PWR(X,Y,Z,V)-\sum_{s=1}^3\epsilon_sPWR(I_sX,I_sY,Z,V)\Big]\\
=\frac14\Big[R(X,Y,Z,V)-\sum_{s=1}^3\epsilon_sR(I_sX,I_sY,Z,V)\Big]+\frac12\sum_{s=1}^3\epsilon_s\omega_s(Z,V)\Big[\tau(X,I_sY)-\tau(I_sX,Y)\Bigr]\\
+\frac{Scal}{32n(n+2)}\Big [ (g\owedge g)(X,Y,Z,V) -\sum_{s=1}^3\epsilon_s
(\omega_s \owedge \omega_s)(X,Y,Z,V) \Big ]\\
 +(g\owedge \mu) (X,Y,Z,V) - \sum_{s=1}^3\epsilon_s(\omega_s\owedge I_s\mu)(X,Y,Z,V).
\end{multline}
\end{prop}
\begin{dfn} We denote the [3]-part of the tensor $PWR$ described in \eqref{qccm} by $W^{pqc}, W^{pqc}:=PWR_{[3]}$
and call it the paraquaternionic contact conformal
curvature.
\end{dfn}
\subsection{Proof of Theorem~\ref{main1}}
The significance  of the tensor  $PWR$ is partially justified by the following
\begin{thrm}
\label{qcinv} The tensor $PWR$ is
invariant under pqc conformal  transformations, i.e. if
\begin{equation*}
\bar\eta=(2h)^{-1}\Psi\eta\quad {\text then} \qquad
2hPWR_{\bar\eta}=PWR_{\eta}, 
\end{equation*}
for any smooth positive function $h$ and any $SO(1,2)$-matrix $\Psi$.
\end{thrm}
\begin{proof}
After a series of  standard computations based on
\eqref{qcw2}, \eqref{Ivan4}, \eqref{qcw3} and  a careful study of the structure of the obtained equation we put it in the following form
\begin{multline}  \label{qcw4}
2hg(\bar R(X,Y)Z,V)-g(R(X,Y)Z,V) 
=-g\owedge M(X,Y,Z,V)+\sum_{s=1}^3\epsilon_s \omega_s\owedge (I_s M)(X,Y,Z,V) \\
-\frac12\sum_{(i,j,k)}\epsilon_i \omega_i(X,Y)\Bigl[M(Z,I_iV)-M(I_iZ,V)-\epsilon_iM(I_jZ,I_kV)+\epsilon_iM(I_kZ,I_jV) %
\Bigr] \\
-g(Z,V)\Bigl[M(X,Y)-M(Y,X)\Bigr]-\sum_{s=1}^3\epsilon_s\omega_s(Z,V)\Bigl[M(X,I_sY)-M(Y,I_sX)\Bigr] \\
+\frac{1}{2n}(tr M)\, \sum_{s=1}^3\epsilon_s\omega_s(X,Y)\omega_s(Z,V) +\frac{1}{2n}\sum_{(i,j,k)}  M_i\Bigl[\omega_j(X,Y)\omega_k(Z,V) -\omega_k(X,Y)\omega_j(Z,V)%
\Bigr],
\end{multline}
where the (0,2) tensor $M$ is given by
\begin{multline}  \label{qcw5}
M(X,Y)=\frac1{2h}\Bigl(\nabla^2h(X,Y)-\frac1{2h}\Bigl[dh(X)dh(Y)-
\sum_{s=1}^3\epsilon_sdh(I_sX)dh(I_sY)+\frac12g(X,Y)|\C h|^2\Bigr]\Bigr),
\end{multline}
and  $tr M=M(e_a,e_a),M_s=M(e_a,I_se_a)$
are its traces. Using \eqref{qcw5} and \eqref{symdh}, we obtain
\begin{equation}  \label{qcw6}
tr M=(2h)^{-1}\Bigl(\triangle_h h-(n+2)h^{-1}|\C h|^2\Bigr), \qquad
M_s=-2n\,h^{-1}dh(\xi_s).
\end{equation}
After taking the traces in \eqref{qcw4},  using \eqref{qcw5} and the
fact that the [3]-component $(\nabla^2h)_{[3]}$ of the horisontal hessian $\nabla^2h$  is symmetric, we obtain
\begin{gather}\label{qcwric}
\overline{Ric}-Ric=4(n+1)M_{[sym]}+6M_{[3]} +\frac{2n+3}{2n}tr M\,g, \quad 
\frac{\overline{Scal}}{2h}-Scal=8(n+2)tr M.
\end{gather}
The $Sp(n,\mathbb R)Sp(1,\mathbb R)$-invariant [-1] and [3] parts of \eqref{qcwric}
are
\begin{gather}\label{qcwp1}
(\overline{Ric}-Ric)_{[-1]}=4(n+1)M_{[sym][-1]}, \quad 
(\overline{Ric}-Ric)_{[3]}=2(2n+5)M_{[3]} + \frac{2n+3}{2n}(tr\,M)\,g.
\end{gather}
The identities in Theorem~\ref{sixtyseven} together with equations \eqref{qcwric}, 
 \eqref{qcwp1} and \eqref{lll} yield
\begin{multline}  \label{mm}
M_{[sym]}=\Bigl(\frac{1}{4(n+1)}\overline{Ric}_{[-1]}+\frac{1}{2(2n+5)}%
\overline{Ric}_{[3]}- \frac{2n+3}{32n(n+2)(2n+5)}\overline{Scal}\,\overline g%
\Bigr) \\
-\Bigl(\frac{1}{4(n+1)}Ric_{[-1]}+\frac{1}{2(2n+5)}Ric_{[3]}- \frac{2n+3}{%
32n(n+2)(2n+5)}Scal\,g\Bigr) \\
=\Bigl[\frac12\overline{\tau}+\overline{\mu}+\frac{\overline{Scal}}{32n(n+2)}\,%
\overline{g}\Bigr]- \Bigl[\frac12\tau+\mu+\frac{Scal}{32n(n+2)}\,g\Bigr]=\bar L-L.
\end{multline}
Next, from \eqref{qcw5} and \eqref{symdh} we obtain
\begin{equation}  \label{mm1}
M(X,Y)=M_{[sym]}(X,Y)+\frac{1}{2h}\sum_{s=1}^3\epsilon_sdh(\xi_s)\omega_s(X,Y).
\end{equation}
Substitute \eqref{mm} into \eqref{mm1}, insert the obtained equality into 
\eqref{qcw4} and use \eqref{qcw6} to complete  the proof of
Theorem~\ref{qcinv}.
\end{proof}
At this point, a combination of Theorem~\ref{qcinv} and
Proposition~\ref{main0} ends the proof of Theorem~\ref{main1}.

The second equality in \eqref{qcwric} combined with the first equation in \eqref{qcw6} yields  
\begin{cor}
On a pqc manifold the pqc scalar curvature transforms under the pqc conformal transformation $\bar\eta=(2h)^{-1}\Psi\eta$ according to the equation
\be\label{pyam}
\overline{Scal}= 2h Scal-8(n+2)^2h^{-1}|\C h|^2+8(n+2)\triangle_hh.
\ee
\end{cor}
The equation \eqref{pyam} constitutes  the sub-hyperbolic Yamabe equation.

\section{Converse problem. Proof of Theorem~\ref{main2}}

Suppose $W^{pqc}=0$, hence $PWR=0$ by Proposition~\ref{main0}. In
order to prove Theorem~\ref{main2} we search for a conformal factor,
such that after a pqc  conformal transformation using this factor the new
pqc structure has  pqc canonical connection with flat horizontal curvature. 
 After we achieve this task, we invoke Proposition~\ref{hflat} and conclude that the given
structure is locally pqc conformal to the flat pqc structure on the paraquaternionic Heisenberg group $G(pH)$.

With these considerations in mind, it is then sufficient to find (locally)  a solution  $h$ of
equation \eqref{mm1} with $M_{[sym]}=-L$. In fact, a  substitution
of \eqref{mm1} in \eqref{qcw4} and an application of the condition
$W^{pqc}=0=PWR$ allows us to see that the pqc structure
$\bar\eta=\frac1{2h}\eta$ has flat pqc canonical connection.

Let us consider the following overdetermined system of partial
differential equations with respect to an unknown function $u$.
\begin{multline}  \label{sist1}
\nabla^2u(X,Y)=-du(X)du(Y)-\sum_{s=1}^3\epsilon_s \Bigl [
du(I_sX)du(I_sY)-du(\xi_s)\omega_s(X,Y)\Bigr ]\\
+\frac12g(X,Y)|\nabla u|^2 -L(X,Y),
\end{multline}
\begin{multline}  \label{add1}
\nabla^2u(X,\xi_i)=\mathbb{B}(X,\xi_i)-L(X,I_idu)+\frac12du(I_iX)|\nabla
u|^2 \\-du(X)du(\xi_i)+\epsilon_idu(I_jX)du(\xi_k)-\epsilon_idu(I_kX)du(\xi_j),
\end{multline}
\begin{gather}  \label{add2}
\nabla^2u(\xi_i,\xi_i)=-\mathbb{B}(\xi_i,\xi_i)+\mathbb{B}%
(I_idu,\xi_i)-\epsilon_i\frac14|\nabla u|^4-(du(\xi_i))^2-\epsilon_k(du(\xi_j))^2-\epsilon_j(du(\xi_k))^2, \\
\nabla^2u(\xi_j,\xi_i)=-\mathbb{B}(\xi_j,\xi_i)+\mathbb{B}%
(I_idu,\xi_j)-2du(\xi_i)du(\xi_j) +\epsilon_k\frac{Scal}{16n(n+2)}du(\xi_k),
\label{add3} \\
\nabla^2u(\xi_k,\xi_i)=-\mathbb{B}(\xi_k,\xi_i)+\mathbb{B}(I_idu,\xi_k)
-2du(\xi_i)du(\xi_k)-\epsilon_j\frac{Scal}{16n(n+2)}du(\xi_j).  \label{add4}
\end{gather}
Here, the tensor $L$ is given by \eqref{lll}, while the tensors
$\mathbb{B}(X,\xi_i)$ and $\mathbb{B}(\xi_i,\xi_j)$ do not depend on
the unknown function $u$ and will be determined later in \eqref{bes}
and \eqref{bst}, respectively. If we make the substitution
\begin{equation*}
2u= \ln h,\qquad 2hdu=dh, \qquad \nabla^2h = 2h\nabla^2u +
4hdu\otimes du,
\end{equation*}
in \eqref{qcw5} we recognize that \eqref{mm1} transforms into
\eqref{sist1}. Therefore, it is sufficient to show that the system
\eqref{sist1}-\eqref{add4} admits (locally) a smooth solution.

The integrability condition of the 
over-determined system \eqref{sist1}-\eqref{add4} is the Ricci identity
\begin{equation}  \label{integr}
\nabla^3u(A,B,C)-\nabla^3u(B,A,C)=-R(A,B,C,du)-\nabla^2u((T(A,B),C), \quad
A,B,C \in \Gamma(TM).
\end{equation}

The proof of Theorem~\ref{main2} will be achieved by considering all
possible cases of \eqref{integr}. It will be presented as a sequel
of subsections, which occupy the rest of this section. The goal is to show that the vanishing of the pqc conformal
tensor $W^{pqc}$ implies the validity of \eqref{integr}, which
guaranties the existence of a local smooth solution to the system
\eqref{sist1}-\eqref{add4}.

We start with 
the next Lemma, which is an application of a standard result in differential geometry. 
\begin{lemma}\label{norma}
 In a neighborhood of any point $p\in M^{4n+3}$
and a  $p\mathbb Q$-orthonormal basis
$$\{X_1(p),X_2(p)=I_1X_1(p)\dots,X_{4n}(p)=I_3X_{4n-3}(p),\xi_1(p),\xi_2(p),\xi_3(p)\}$$ of the tangential space at $p$,
there exists a $p\mathbb Q$ - orthonormal frame field
\begin{equation*}
\begin{split}
\{X_1,X_2=I_1X_1,\dots,X_{4n}=I_3X_{4n-3},\xi_1,\xi_2,\xi_3\},
X_{a_|p}=X_a(p),
 \xi_{s_|p}=\xi_s(p),
  \end{split}
  \end{equation*}
   such that the
connection 1-forms of the canonical connection are all zero at the point p, i.e., we have
\begin{equation}\label{norm}
(\nabla_{X_a}X_b)_|p=(\nabla_{\xi_i}X_b)_|p=(\nabla_{X_a}\xi_t)_|p=(\nabla_{\xi_t}\xi_s)_|p=0,
\end{equation}
for $a,b =1,\dots,4n,  s,t,r=1,2,3.$
In particular,
$$((\nabla_{X_a}I_s)X_b){_|p}=((\nabla_{X_a}I_s)\xi_t){_|p}=((\nabla_{\xi_t}I_s)X_b){_|p}
=((\nabla_{\xi_t}I_s)\xi_r){_|p}=0.$$
\end{lemma}
\begin{proof}
Since $\nabla$ preserves the splitting $H\oplus V$ we can apply the standard arguments
for the existence of a normal frame with respect to a metric connection (see e.g.
\cite{Wu}). We sketch the proof for completeness.

Let $\{\tilde X_1,\dots,\tilde X_{4n},\tilde\xi_1,\tilde\xi_2,\tilde\xi_3\}$ be a
$p\mathbb Q$-orthonormal basis around p such that $\tilde X_{a_|p}=X_a(p),\quad
\tilde\xi_{i_|p}=\xi_i(p)$. We want to find a modified  frame
$X_a=o^b_a\tilde X_b, \quad \xi_i=o^j_i\tilde\xi_j,$ which
satisfies the normality conditions of the lemma.

Let $\varpi$ be the $sp(n,\mathbb R)\oplus sp(1,\mathbb R)$-valued connection 1-forms with respect to the frame
$\{\tilde X_1,\dots,\tilde X_{4n},\tilde\xi_1,\tilde\xi_2,\tilde\xi_3\}$, 
$$\nabla\tilde X_b=\varpi^c_b\tilde X_c, \quad
\nabla\tilde\xi_s=\varpi^t_s\tilde\xi_t, \quad B\in\{\tilde X_1,\dots,\tilde
X_{4n},\tilde\xi_1,\tilde\xi_2,\tilde\xi_3\}.$$ Let $\{x^1,\dots,x^{4n+3}\}$ be a
coordinate system around p, such that $$\frac{\partial}{\partial x^a}(p)=X_a(p),\quad
\frac{\partial}{\partial x^{4n+t}}(p)=\xi_t(p), \quad a=1,\dots,4n, \quad t=1,2,3.$$ One
can easily check that the matrices with entries
$$o^b_a=exp\left(-\sum_{c=1}^{4n+3}\varpi^b_a(\frac{\partial}{\partial x^c})_{|p}x^c\right)\in Sp(n,\mathbb R),\
o^s_t=exp\left(-\sum_{c=1}^{4n+3}\varpi^s_t(\frac{\partial}{\partial
x^c})_{|p}x^c\right)\in Sp(1,\mathbb R)$$ are the desired matrices making the identities
\eqref{norm} true.

Next, the last identity in the lemma is a consequence of the fact that the choice of the
orthonormal basis of $V$ does not depend on the action of $SO(1,2)$ on $V$ combined with \eqref{conef}.
\end{proof}
\begin{dfn}\label{d:normal frame}
We refer to the orthonormal frame constructed in Lemma~\ref{norma} as a \emph{pqc-normal
frame}.
\end{dfn}
Since \eqref{integr} is
$Sp(n,\mathbb R)Sp(1,\mathbb R)$-invariant it is sufficient to check it in a pqc-normal frame. 

From now on
the frame  $\{e_1,e_2=I_1e_1,e_3=I_2e_1,e_4=I_3e_1 \dots,
e_{4n}=I_3e_{4n-3}, \xi_1, \xi_2, \xi_3 \}$
is a fixed pqc-normal frame at a fixed point $p\in M$.

\subsection{Case 1: $X,Y,Z \in H$. Integrability condition ~(\ref{inte})}\hfill

The equation \eqref{integr} on $H$  takes the form
\begin{multline}  \label{inteh}
\nabla^3u(Z,X,Y)-\nabla^3u(X,Z,Y)=-R(Z,X,Y,du) \\
+ 2\omega_1(Z,X)\nabla^2u(\xi_1,Y)+2\omega_2(Z,X)\nabla^2u(\xi_2,Y)-2\omega_3(Z,X)\nabla^2u(\xi_3,Y),
\end{multline}
where we have used \eqref{torha}. The identity $d^2u=0$ gives
\begin{equation}  \label{comutat}
\nabla^2u(X,\xi_s)-\nabla^2u(\xi_s,X)=du(T(\xi_s,X))=T(\xi_s,X,du).
\end{equation}
After we take a covariant derivative of \eqref{sist1} along $Z\in
H$, substitute the derivatives from \eqref{sist1} and \eqref{add1},
then anti-commute the covariant derivatives, substitute the result
in \eqref{inteh}, use \eqref{qcwdef} with $PWR=0$, \eqref{comutat}, \eqref{lll} and the properties of the torsion
described in  Proposition~\ref{torb}, we obtain by series of standard calculations, that the integrability condition in this
case is
\begin{multline}  \label{inte}
(\nabla_ZL)(X,Y)-(\nabla_XL)(Z,Y)\\
=-\sum_{s=1}^3\epsilon_s\Bigl[\omega_s(Z,Y)\mathbb{B}%
(X,\xi_s)- \omega_s(X,Y)\mathbb{B}(Z,\xi_s)+2\omega_s(Z,X)\mathbb{B}%
(Y,\xi_s) \Bigr].
\end{multline}
Now, we determine the tensors $\mathbb{B}(X,\xi_s)$. The traces in \eqref{inte}  give the next sequence of
equalities
\begin{equation}  \label{bes}
\begin{aligned}
&(\nabla_{e_a}L)(I_ie_a,I_iX)=(4n+1)\mathbb{B}(I_iX,\xi_i)+\epsilon_k\mathbb{B}%
(I_jX,\xi_j)+\epsilon_j\mathbb{B}(I_kX,\xi_k); \\
&\sum_{s=1}^3\epsilon_s\mathbb{B}(I_sX,\xi_s)=\frac13\Bigl[(
\nabla_{e_a}L)(e_a,X)-\nabla_Xtr\,L\Bigr]=\frac1{4n-1}\sum_{s=1}^3\epsilon_s(
\nabla_{e_a}L)(I_se_a,I_sX); \\
&\mathbb{B}(X,\xi_i)=\frac1{2(2n+1)}\Bigl[(\nabla_{e_a}L)(I_ie_a,X)+\frac13%
\Bigl((\nabla_{e_a}L)(e_a,I_iX)-\nabla_{I_iX}tr\,L\Bigl)\Bigl],
\end{aligned}
\end{equation}
where the second equality in \eqref{bes} is precisely equivalent to %
\eqref{div}.
\begin{lemma}
\label{prost} The condition \eqref{inte} is equivalent to
\begin{equation}\label{con111}
(\nabla_ZL)(X,Y)-(\nabla_XL)(Z,Y)=0\qquad {\text mod} \quad g, \omega_1,\omega_2,\omega_3.
\end{equation}
\end{lemma}
\begin{proof}
Observe  the 
 cyclic sum $
\sum_{(Z,X,Y)}[(\nabla_ZL)(X,Y)-(\nabla_XL)(Z,Y)]=0.$ 
The condition \eqref{con111} implies
\begin{multline}  \label{alg1}
(\nabla_ZL)(X,Y)-(\nabla_XL)(Z,Y)= g(Z,Y)C(X)-g(X,Y)C(Z) \\
-\sum_{s=1}^3\epsilon_s\Bigl[\omega_s(Z,Y) \mathbb{B}(X,\xi_s)- \omega_s(X,Y)\mathbb{B}%
(Z,\xi_s)+2\omega_s(Z,X) B\mathbb{(}Y,\xi_s) \Bigr],
\end{multline}
for some tensors $C(X),B(X,\xi_s)$.   The traces
in \eqref{alg1} yield
\begin{equation*}
\begin{aligned}
&(\nabla_{e_a}L)(I_ie_a,I_iX)=(4n+1) \mathbb{B}(I_iX,\xi_i)+\epsilon_k
\mathbb{B}
(I_jX,\xi_j)+\epsilon_j \mathbb{B}(I_kX,\xi_k) +\epsilon_iC(X) \quad \text{implying} \\
& \sum_{s=1}^3\epsilon_s(\nabla_{e_a}L)(I_se_a,I_sX)=\sum_{s=1}^3 \epsilon_s(4n-1)
\mathbb{B} (I_sX,\xi_s)+3C(X).
\end{aligned}
\end{equation*}
The latter together with the second equality in
\eqref{bes} shows  $C(X)=0$.
\end{proof}
\begin{prop}
\label{integrmain} If $W^{pqc}=0$  then the
condition \eqref{inte} holds.
\end{prop}
\begin{proof} 
The  second Bianchi identity
\begin{equation}  \label{secb}
\sum_{(A,B,C)}\Big\{(\nabla_AR)(B,C,D,E)+R(T(A,B),C,D,E)\Big\}=0,
\end{equation}
 combined with  \eqref{torha} yields
\begin{equation}  \label{bi20}
\sum_{(X,Y,Z)} \Bigl[(\nabla_XR)(Y,Z,V,W)-2\sum_{s=1}^3\epsilon_s\omega_s(X,Y)R(%
\xi_s,Z,V,W)\Bigr]=0.
\end{equation}
The trace in \eqref{bi20} leads to
\begin{multline}\label{bi2ric}
(\C_{e_a}R)(X,Y,Z,e_a)=(\C_XRic)(Y,Z)-(\C_YRic)(X,Z)\\
+2\sum_{s=1}^3\epsilon_s\Big[ R(\xi_s,Y,Z,I_sX)-R(\xi_s,X,Z,I_sY)+\f_s(X,Y)Ric(\xi_s,Z) \Big].
\end{multline}
We use $W^{pqc}=0$,  \eqref{bi2ric} and apply \eqref{qcwdef} to
calculate
\begin{multline}\label{n=11}
(\nabla_{e_a}R)(X,Y,Z,e_a)=-(\nabla_YL)(X,Z)+(\nabla_XL)(Y,Z)\\
+\sum_{s=1}^3\epsilon_s\Big[(\nabla_{I_sX}L)(Y,I_sZ)-(\nabla_{I_sY}L)(X,I_sZ)-(\nabla_{I_sZ}L)(X,I_sY)+
(\nabla_{I_sZ}L)(I_sX,Y)\Big]
\quad {\text mod} \quad g,\omega_s.
\end{multline}
By substituting \eqref{t01}, \eqref{u01} into \eqref{vert1} we get

\begin{multline}\label{n=12}
2\sum_{s=1}^3\epsilon_s\Big[R(\xi_s,Y,Z,I_sX)-R(\xi_s,X,Z,I_sY)\Big]\\=
-\sum_{s=1}^3\epsilon_s\Big[(\nabla_{I_sY}L)(X,I_sZ)-(\nabla_{I_sX}L)(Y,I_sZ)+(\nabla_{I_sY}L)(I_sX,Z)-
(\nabla_{I_sX}L)(I_sY,Z)\Big]\\+
\frac32\sum_{s=1}^3\Big[(\nabla_{Y}L)(X,Z)-(\nabla_{X}L)(Y,Z)-\epsilon_s(\nabla_{Y}L)(I_sX,I_sZ)+\epsilon_s
(\nabla_{X}L)(I_sY,I_sZ)\Big] \quad {\text mod} \quad g,\omega_s.
\end{multline}
The second Bianchi identity gives
$\sum_{(X,Y,Z)}(\nabla_X\rho_i)(Y,Z)=0\quad {\text mod}\quad
g,\omega_s.$ Use \eqref{ricis}  to see
\begin{multline}  \label{neweq}
3\Bigl((\nabla_YL)(X,Z)-(\nabla_XL)(Y,Z)\Bigr)-\sum_{s=1}^3\epsilon_s\Bigl(%
(\nabla_YL)(I_sX,I_sZ)-(\nabla_XL)(I_sY,I_sZ)\Bigr) \\
- \sum_{s=1}^3\epsilon_s\Bigl[(\nabla_{I_sZ}L)(X,I_sY)-(\nabla_{I_sZ}L)(I_sX,Y)\Bigr]%
=0 \quad {\text mod} \quad g,\omega_s.
\end{multline}
A substitution of \eqref{n=11}, \eqref{n=12}, \eqref{neweq} and
\eqref{ricis} in \eqref{bi2ric} shows, after  standard calculations, that
\begin{multline}  \label{qcbi1}
(4n+3)\Bigl[(\nabla_YL)(X,Z)-(\nabla_XL)(Y,Z)\Bigr]-
\sum_{s=1}^3\epsilon_s\Bigl[(\nabla_{I_sY}L)(I_sX,Z)-(\nabla_{I_sX}L)(I_sY,Z)\Bigr]\\-
2\sum_{s=1}^3\epsilon_s\Bigl[(\nabla_YL)(I_sX,I_sZ)-(\nabla_{I_sX}L)(Y,I_sZ)+(\nabla_{I_sY}L)(X,I_sZ)-(\nabla_XL)(I_sY,I_sZ)
\Bigr]
=0\quad {\text mod} \quad g,\omega_s.
\end{multline}
Taking the [3]-component with respect to $X,Y$ in \eqref{qcbi1}, we obtain
\begin{equation}\label{n=13}
(\nabla_YL)(X,Z)-(\nabla_XL)(Y,Z)-\sum_{s=1}^3\epsilon_s\Big[
(\nabla_{I_sY}L)(I_sX,Z)-(\nabla_{I_sX}L)(I_sY,Z)\Big]
=0\quad {\text mod} \quad g,\omega_s.
\end{equation}
A substitution of  \eqref{n=13} in \eqref{qcbi1} gives
\begin{multline}\label{qcbin1}
2n\Bigl[(\nabla_YL)(X,Z)-(\nabla_XL)(Y,Z)\Bigr]-
\sum_{s=1}^3\epsilon_s\Bigl[(\nabla_{I_sY}L)(X,I_sZ)-(\nabla_{X}L)(I_sY,I_sZ)\Bigr]\\+
(\nabla_YL)(X,Z)-(\nabla_XL)(Y,Z)-\sum_{s=1}^3\epsilon_s\Bigl[(\nabla_YL)(I_sX,I_sZ)-(\nabla_{I_sX}L)(Y,I_sZ)
\Bigr]
=0\quad {\text mod} \quad g,\omega_s.
\end{multline}
Taking the [-1]-component with respect to $X,Z$ of \eqref{qcbin1},
calculated with the help of \eqref{n=13}, yields
\begin{multline}  \label{qcbin2}
(6n-1)\Bigl[(\nabla_YL)(X,Z)-(\nabla_XL)(Y,Z)\Bigr]-
4\sum_{s=1}^3\epsilon_s\Bigl[(\nabla_{I_sY}L)(X,I_sZ)-(\nabla_{X}L)(I_sY,I_sZ)\Bigr]\\
+(2n+1)\sum_{s=1}^3\epsilon_s\Bigl[(\nabla_YL)(I_sX,I_sZ)-(\nabla_{I_sX}L)(Y,I_sZ)
\Bigr]
=0\quad {\text mod} \quad g,\omega_s.
\end{multline}
The equations \eqref{qcbin1} and \eqref{qcbin2} lead to
\begin{multline*}
(\nabla_YL)(X,Z)-(\nabla_XL)(Y,Z)-
\sum_{s=1}^3\epsilon_s\Bigl[(\nabla_YL)(I_sX,I_sZ)-(\nabla_{I_sX}L)(Y,I_sZ)\Bigr]
=0\quad {\text mod} \quad g,\omega_s.
\end{multline*}
The latter  and \eqref{qcbin1} imply
\begin{equation}  \label{qcbi5}
(2n-1)\Big[(\nabla_YL)(X,Z)-(\nabla_XL)(Y,Z)\Big] =0 \quad {\text
mod} \quad g,\omega_s
\end{equation}
and Lemma~\ref{prost} completes the proof of \eqref{inte}.
\end{proof}

\subsection{Case 2: $Z,X \in H,\quad \xi_i\in V$. Integrability condition
~(\ref{inte1})}\hfill

In this second case \eqref{integr} turns into
\begin{multline}  \label{intehxi}
\nabla^3u(Z,X,\xi_{i})-\nabla^3u(X,Z,\xi_{i})=-R(Z,X,\xi_{i},du)-\nabla^2u(T(Z,X),\xi_{i})= \\
-2\epsilon_idu(\xi_{j})\rho_{k}(Z,X)+2\epsilon_idu(\xi_{k})\rho_{j}(Z,X) \\
+2\epsilon_i\omega_{i}(Z,X)\nabla^2u(\xi_{i},\xi_{i})+2\epsilon_j\omega_{j}(Z,X)\nabla^2u(\xi_{j},\xi_{i})+2\epsilon_k\omega_{k}(Z,X)\nabla^2u(\xi_{k},\xi_{i}),
\end{multline}
after using \eqref{torha} and \eqref{rhov}.  
Take a covariant
derivative of \eqref{add1} along $Z\in H$, substitute into the obtained equality \eqref{sist1} and \eqref{add1} and anti-commute the covariant derivatives to get
\begin{multline}  \label{intver1}
\nabla^3u(Z,X,\xi_{i})-\nabla^3u(X,Z,\xi_{i})=(\nabla_Z\mathbb{B}%
)(X,\xi_{i})-(\nabla_X\mathbb{B})(Z,\xi_{i}) \\
-(\nabla_ZL)(X,I_{i}du)+(\nabla_XL)(Z,I_{i}du)
-L(X,\nabla_ZI_{i}du)+L(Z,\nabla_XI_{i}du) \\
+ {\text other \quad terms \quad comming \quad from \quad the \quad use
\quad of \quad \eqref{sist1} \quad and \quad \eqref{add1}}.
\end{multline}
Next, we substitute \eqref{intver1} into \eqref{intehxi}, use \eqref{inte}, already proved in Proposition~\ref{integrmain}, together with
\eqref{add2}, \eqref{add3}, \eqref{add4} and the second equation in \eqref{ricis}, perform
some basic calculation and  arrive at
\begin{equation}  \label{inte1}
(\nabla_Z\mathbb{B})(X,\xi_t)-(\nabla_X\mathbb{B})(Z,%
\xi_t)-L(Z,I_tL(X))+L(X,I_tL(Z))= -2\sum_{s=1}^3\epsilon_s\mathbb{B}(\xi_s,\xi_t)%
\omega_s(Z,X),
\end{equation}
which is the integrability condition in this case. The
functions $\mathbb{B}(\xi_s,\xi_t)$ are uniquely determined by
\begin{equation}  \label{bst}
\mathbb B(\xi_s,\xi_t)=\frac1{4n}\Bigl[ (\nabla_{e_a}\mathbb
B)(I_se_a,\xi_t) +L(e_a,e_b)L(I_te_a,I_se_b)\Bigr].
\end{equation}
\begin{prop}
\label{integrmain1} If $W^{qc}=0$ then
the condition \eqref{inte1} holds.
\end{prop}
\begin{proof}
To prove the assertion it is sufficient to
 show that the left hand side of \eqref{inte1} vanishes $mod\quad \omega_s$.
Differentiating \eqref{inte} and taking the corresponding traces
yield
\begin{multline}\label{int01}
(\nabla^2_{e_a,I_{i}e_a}L)(X,Y)-(\nabla^2_{e_a,X}L)(I_{i}e_a,Y)=
-(\nabla_Y
\mathbb{B})(X,\xi_{i})-2(\nabla_X\mathbb{B})(Y,\xi_{i}) \\
-\epsilon_i\Big[(\nabla_{I_{k}Y}\mathbb{B})(X,\xi_{j})+2(\nabla_{I_{k}X}\mathbb{B})(Y,\xi_{j})-
(\nabla_{I_{j}Y}\mathbb{B})(X,\xi_{k})-2(\nabla_{I_{j}X}\mathbb{B})(Y,\xi_{k})\Big];
\end{multline}
\begin{multline}\label{int02}
(\nabla^2_{e_a,X}L)(I_{i}e_a,Y)-(\nabla^2_{e_a,Y}L)(I_{i}e_a,X)=(\nabla_X\mathbb{
B})(Y,\xi_{i})-(\nabla_Y\mathbb{B})(X,\xi_{i}) \\
-\epsilon_i\Big[(\nabla_{I_{k}Y}\mathbb{B})(X,\xi_{j})-(\nabla_{I_{k}X}\mathbb{B}
)(Y,\xi_{j})-(\nabla_{I_{j}Y}\mathbb{B})(X,\xi_{k})+
(\nabla_{I_{j}X}\mathbb{B} )(Y,\xi_{k})\Big]\quad {\text mod}\quad
\omega_s;
\end{multline}
 \begin{multline}\label{int03}
(\nabla^2_{X,e_a}L)(I_{i}e_a,Y)=(4n+1)(\nabla_X\mathbb{B})(Y,\xi_{i})+\epsilon_i(\nabla_X
\mathbb{B})(I_{k}Y,\xi_{j})-\epsilon_i
(\nabla_X\mathbb{B})(I_{j}Y,\xi_{k})\quad {\text mod} \quad \omega_s;
\end{multline}
\begin{multline}\label{int04}
-\nabla^2_{X,I_{i}Y}tr\,L+(\nabla^2_{X,e_a}L)(e_a,I_{i}Y)=
3(\nabla_X\mathbb{B}
)(Y,\xi_{i})-3\epsilon_i(\nabla_X\mathbb{B})(I_{k}Y,\xi_{j})+3\epsilon_i(\nabla_X\mathbb{B}
)(I_{j}Y,\xi_{k}).
\end{multline}
We obtain from equalities \eqref{int02} and \eqref{int03} that
\begin{multline}  \label{int05}
\Bigl[\nabla^2_{X,e_a}-\nabla^2_{e_a,X}\Bigr]L(I_{i}e_a,Y)+\Bigl[%
\nabla^2_{e_a,Y}-\nabla^2_{Y,e_a}\Bigr]L(I_{i}e_a,X)\\
= 4n\Bigl[(\nabla_X\mathbb{B})(Y,\xi_{i})-(\nabla_Y\mathbb{B})(X,\xi_{i})\Bigr] \\
+\epsilon_i\Bigl[(\nabla_X\mathbb{B})(I_{k}Y,\xi_{j})+(\nabla_{I_{k}Y}\mathbb{B}%
)(X,\xi_{j})-(\nabla_Y\mathbb{B})(I_{k}X,\xi_{j})- (\nabla_{I_{k}X}\mathbb{B}%
)(Y,\xi_{j})\Bigr] \\
-\epsilon_i\Bigl[(\nabla_X\mathbb{B})(I_{j}Y,\xi_{k})+(\nabla_{I_{j}Y}\mathbb{B}%
)(X,\xi_{k})-(\nabla_Y\mathbb{B})(I_{j}X,\xi_{k})- (\nabla_{I_{j}X}\mathbb{B}%
)(Y,\xi_{k})\Bigr] \qquad {\text mod}\quad \omega_s.
\end{multline}
On the other hand, the Ricci identities
\begin{multline}  \label{int006}
\Bigl[\nabla^2_{X,e_a}-\nabla^2_{e_a,X}\Bigr]%
L(I_{i}e_a,Y)=-R(X,e_a,Y,e_b)L(e_b,I_{i}e_a)-4n\zeta_{i}(X,e_a)L(Y,e_a) \\
+2(\nabla_{\xi_{i}}L)(X,Y)+2\epsilon_i(\nabla_{\xi_{j}}L)(I_{k}X,Y)-2\epsilon_i(\nabla_{%
\xi_{k}}L)(I_{j}X,Y),
\end{multline}
 the first Bianchi identity \eqref{bian1} and Proposition~\ref{torb} imply
\begin{multline}  \label{int06}
\Bigl[\nabla^2_{X,e_a}-\nabla^2_{e_a,X}\Bigr]L(I_{i}e_a,Y)+\Bigl[%
\nabla^2_{e_a,Y}-\nabla^2_{Y,e_a}\Bigr]L(I_{i}e_a,X)= \\
2\epsilon_i\Bigl[(\nabla_{\xi_{j}}L)(I_{k}X,Y)-(\nabla_{\xi_{j}}L)(X,I_{k}Y)\Bigr] -2\epsilon_i\Bigl[%
(\nabla_{\xi_{k}}L)(I_{j}X,Y)-(\nabla_{\xi_{k}}L)(X,I_{j}Y)\Bigr] \\
+2T(\xi_{i},Y,e_a)L(X,e_a)+2\epsilon_iT(\xi_{j},Y,e_a)L(I_{k}X,e_a)-2\epsilon_iT(%
\xi_{k},Y,e_a)L(I_{j}X,e_a) \\
-2T(\xi_{i},X,e_a)L(Y,e_a)-2\epsilon_iT(\xi_{j},X,e_a)L(I_{k}Y,e_a)+2\epsilon_iT(%
\xi_{k},X,e_a)L(I_{j}Y,e_a) \\
-R(X,Y,e_a,e_b)L(e_b,I_{i}e_a)-4n[\zeta_{i}(X,e_a)L(Y,e_a)-%
\zeta_{i}(Y,e_a)L(X,e_a)] \qquad {\text mod}\quad \omega_s.
\end{multline}
The second equality in \eqref{ricis} and a suitable contraction in
the second Bianchi identity give the next two equations valid mod $
\omega_s$
\begin{equation}  \label{xirho}
\begin{aligned}
(\nabla_{\xi_{j}}L)(X,I_{k}Y) & -(\nabla_{\xi_{j}}L)(I_{k}X,Y)=
(\nabla_{\xi_{j}}\rho_{k})(X,Y) \\
& =(\nabla_X\rho_{k})(\xi_{j},Y)-(\nabla_Y\rho_{k})(\xi_{j},X)-\rho_{k}(T(\xi_{j},X),Y)+%
\rho_{k}(T(\xi_{j},Y),X); \\
 (\nabla_{\xi_{k}}L) (X,I_{j}Y) & -(\nabla_{\xi_{k}}L)(I_{j}X,Y)=
(\nabla_{\xi_{k}}\rho_{j})(X,Y) \\
& =(\nabla_X\rho_{j})(\xi_{k},Y)-(\nabla_Y\rho_{j})(\xi_{k},X)-\rho_{j}(T(\xi_{k},X),Y)+%
\rho_{j}(T(\xi_{k},Y),X) .
\end{aligned}
\end{equation}
A substitution of \eqref{t01}, \eqref{u01} into
\eqref{vert023} together with \eqref{inte} and an
application of \eqref{bes} give 
\begin{lemma}
\label{l:rho-ta} We have the following formulas for the Ricci 2-forms
\begin{equation}  \label{verf1}
\begin{aligned} \rho_{k}(\xi_{i},X) & =
-\epsilon_j\mathbb B(X,\xi_{j})-\mathbb B(I_{k}X,\xi_{i}),\hskip.7truein \rho_{i}(\xi_{k},X) =
\epsilon_j\mathbb B(X,\xi_{j})-\mathbb B(I_{i}X,\xi_{k}),\\ \rho_{i}(X,\xi_{i}) & = \epsilon_i\frac
{1}{4n}d(trL)(X) + \mathbb B(I_{i}X,\xi_{i}). \end{aligned}
\end{equation}
\end{lemma}
The identity $\C^2u(\xi_i,\xi_j)-\C^2u(\xi_j,\xi_i)=-T(\xi_i,\xi_j,du)$, \eqref{add3}, \eqref{lll}, \eqref{torv}, \eqref{vertor} and \eqref{verf1} imply
\begin{lemma}\label{bijn=1}
The tensors $B(\xi_i,\xi_j)$ are symmetric, $B(\xi_i,\xi_j)=B(\xi_j,\xi_i)$.
\end{lemma}
When we take the covariant derivative of \eqref{verf1}, substitute
the obtained equalities together with
 \eqref{int06}, \eqref{xirho} 
in \eqref{int05}, we derive the formula
\begin{multline}  \label{int08}
(4n+2)\Bigl[(\nabla_X\mathbb{B})(Y,\xi_{i})-(\nabla_Y\mathbb{B})(X,\xi_{i})\Bigr]%
-\epsilon_j \Bigl[(\nabla_{I_{j}X}\mathbb{B})(I_{j}Y,\xi_{i})-(\nabla_{I_{j}Y}\mathbb{B}%
)(I_{j}X,\xi_{i})\Bigr] \\
-\epsilon_k\Bigl[(\nabla_{I_{k}X}\mathbb{B})(I_{k}Y,\xi_{i})-(\nabla_{I_{k}Y}\mathbb{B}%
)(I_{k}X,\xi_{i})\Bigr]=F(X,Y) \quad {\text mod} \quad \omega_s,
\end{multline}
where the (0,2)-tensor $F$ is defined by
\begin{multline}  \label{int09}
F(X,Y)= -R(X,Y,e_a,e_b)L(e_b,I_{i}e_a)-4n\Bigl[\zeta_{i}(X,e_a)L(Y,e_a)-%
\zeta_{i}(Y,e_a)L(X,e_a)\Bigr] \\
+2T(\xi_{i},Y,e_a)L(X,e_a)+2\epsilon_i\Big[T(\xi_{j},Y,e_a)L(I_{k}X,e_a)-T(%
\xi_{k},Y,e_a)L(I_{j}X,e_a)\Big] \\
-2T(\xi_{i},X,e_a)L(Y,e_a)-2\epsilon_i\Big[T(\xi_{j},X,e_a)L(I_{k}Y,e_a)-T(%
\xi_{k},X,e_a)L(I_{j}Y,e_a)\Big] \\
-\epsilon_i\Big[ \rho_{j}(T(\xi_{k},X),Y)-\rho_{j}(T(\xi_{k},Y),X)\Big]
-\epsilon_k\Big[\rho_{j}(T(\xi_{k},I_{j}X),I_{j}Y)-\rho_{j}(T(\xi_{k},I_{j}Y),I_{j}X)\Big] \\
+\epsilon_i\Big[\rho_{k}(T(\xi_{j},X),Y)-\rho_{k}(T(\xi_{j},Y),X)\Big] +\epsilon_j\Big[
\rho_{k}(T(\xi_{j},I_{k}X),I_{k}Y)-\rho_{k}(T(\xi_{j},I_{k}Y),I_{k}X)\Big].
\end{multline}
Solving for
$(\nabla_X\mathbb{B})(Y,\xi_{i})-(\nabla_Y\mathbb{B})(X,\xi_{i})$ from \eqref{int08} we
obtain
\begin{multline}  \label{fininte2}
16n(n+1)(2n+1)\Bigl[(\nabla_X\mathbb{B})(Y,\xi_{i})-(\nabla_Y\mathbb{B}%
)(X,\xi_{i})\Bigr] \\
=(8n^2+8n+1)F(X,Y)-\epsilon_iF(I_{i}X,I_{i}Y)+(2n+1)\Bigl[\epsilon_jF(I_{j}X,I_{j}Y)+\epsilon_kF(I_{k}X,I_{k}Y)\Bigr] %
\quad {\text mod} \quad \omega_s.
\end{multline}
The condition $W^{pqc}=0$ and \eqref{qcwdef} give
\begin{multline}  \label{curvl}
-R(X,Y,e_a,e_b)L(I_{i}e_a,e_b)=4L(X,e_a)L(Y,I_{i}e_a)-2L(X,e_a)L(I_{i}Y,e_a) \\
+2L(I_{i}X,e_a)L(Y,e_a) -2\epsilon_i\Big[L(X,e_a)L(I_{j}Y,I_{k}e_a)-L(I_{k}X,e_a)L(Y,I_{j}e_a) \\
-L(X,e_a)L(I_{k}Y,I_{j}e_a)+L(I_{j}X,e_a)L(Y,I_{k}e_a)\Big] -tr\,L\Big[%
L(X,I_{i}Y)-L(I_{i}X,Y)\Big].
\end{multline}
Using \eqref{ricis}, we obtain from 
 \eqref{curvl} 
 \begin{multline}  \label{curzetl}
-R(X,Y,e_a,e_b)L(I_{i}e_a,e_b)-4n\Bigl[\zeta_{i}(X,e_a)L(Y,e_a)-%
\zeta_{i}(Y,e_a)L(X,e_a)\Bigr] \\
=-(8n-1)L(X,e_a)L(Y,I_{i}e_a)-\frac1{2n}(tr L)\Big[L(X,I_{i}Y)-L(I_{i}X,Y)\Big] \\
+\frac12\Big[L(Y,e_a)L(I_{i}X,e_a)-L(X,e_a)L(I_{i}Y,e_a)\Big] +\frac32\epsilon_i\Big[L(Y,e_a)L(I_{j}X,I_{k}e_a)-L(X,e_a)L(I_{j}Y,I_{k}e_a)\Big] \\-
\frac32\epsilon_i\Big[L(Y,e_a)L(I_{k}X,I_{j}e_a)-L(X,e_a)L(I_{k}Y,I_{j}e_a)\Big].
\end{multline}
Since $\rho_s$ is  a (1,1)-form with respect to $I_s$ (see
Proposition~\ref {sixtyseven}), we have
$$\rho_{j}(T(\xi_{k},I_{j}X),I_{j}Y)=\rho_{j}(e_a,I_{j}Y)T(\xi_{k},I_{j}X,e_a)=%
\rho_{j}(e_a,Y)T(\xi_{k},I_{j}X,I_{j}e_a).$$ 
Using
\eqref{ricis} we obtain 
\begin{multline}  \label{rhot11}
-\epsilon_i\rho_{j}(T(\xi_{k},X),Y)-\epsilon_k\rho_{j}(T(\xi_{k},I_{j}X),I_{j}Y)+\epsilon_i\rho_{k}(T(\xi_{j},X),Y)+\epsilon_j
\rho_{k}(T(\xi_{j},I_{k}X),I_{k}Y)\\
-2T(\xi_{i},X,e_a)L(Y,e_a)-2\epsilon_i\Big[T(\xi_{j},X,e_a)L(I_{k}Y,e_a)-2T(%
\xi_{k},X,e_a)L(I_{j}Y,e_a)\Big] \\
= L(e_a,Y)\Big[-\epsilon_iT(\xi_{k},X,I_{j}e_a)+\epsilon_iT(\xi_{k},I_{j}X,e_a)+\epsilon_iT(\xi_{j},X,I_{k}e_a)-\epsilon_iT(%
\xi_{j},I_{k}X,e_a)-2T(\xi_{i},X,e_a) \Big] \\
+ L(e_a,I_{j}Y)\Big[-\epsilon_kT(\xi_{k},I_{j}X,I_{j}e_a)+\epsilon_iT(\xi_{k},X,e_a)\Big]-L(e_a,I_{k}Y)\Big[%
-\epsilon_jT(\xi_{j},I_{k}X,I_{k}e_a)+\epsilon_iT(\xi_{j},X,e_a)\Big] \\
+\frac{\epsilon_i}{2n}tr\,L\Big[-T(\xi_{k},X,I_{j}Y)+T(\xi_{k},I_{j}X,Y)-T(\xi_{j},I_{k}X,Y)+T(%
\xi_{j},X,I_{k}Y)\Big]\\
= \frac1{2n}tr\,L.\,L(X,I_{i}Y)+\frac12 L(Y,e_a)\Big[5L(X,I_{i}e_a)
-L(I_{i}X,e_a)-\epsilon_iL(I_{j}X,I_{k}e_a)+\epsilon_iL(I_{k}X,I_{j}e_a) \Big] \\
-\epsilon_iL(X,e_a)\Big[L(I_{k}Y,I_{j}e_a) -L(I_{j}Y,I_{k}e_a)\Big]
+\epsilon_jL(I_{j}X,e_a)L(I_{j}Y,I_{i}e_a) +\epsilon_kL(I_{k}X,e_a)L(I_{k}Y,I_{i}e_a).
\end{multline}
The last four lines in \eqref{int09} equal the skew symmetric sum of %
\eqref{rhot11}, which is given by
\begin{multline}  \label{rhot111}
-5L(X,e_a)L(Y,I_{i}e_a) -\frac12\Big[L(Y,e_a)L(I_{i}X,e_a)-L(X,e_a)L(I_{i}Y,e_a)%
\Big] \\
- \frac32\epsilon_i\Big[L(Y,e_a)L(I_{j}X,I_{k}e_a)-L(X,e_a)L(I_{j}Y,I_{k}e_a)\Big]+ \frac32%
\epsilon_i\Big[L(Y,e_a)L(I_{k}X,I_{j}e_a)-L(X,e_a)L(I_{k}Y,I_{j}e_a)\Big] \\
+\frac1{2n}tr\,L\Big[L(X,I_{i}Y)-L(I_{i}X,Y)\Big]
+2\epsilon_jL(I_{j}X,e_a)L(I_{j}Y,I_{i}e_a) +2\epsilon_kL(I_{k}X,e_a)L(I_{k}Y,I_{i}e_a).
\end{multline}
A substitution of \eqref{curzetl} and \eqref{rhot111} in
\eqref{int09} yields
\begin{equation}  \label{fff}
F(X,Y)=2\epsilon_jL(I_{j}X,e_a)L(I_{j}Y,I_{i}e_a)+
2\epsilon_kL(I_{k}X,e_a)L(I_{k}Y,I_{i}e_a)-4(2n+1)L(X,e_a)L(Y,I_{i}e_a).
\end{equation}
Inserting \eqref{fff} into \eqref{fininte2}  completes the proof of
\eqref{inte1}.
\end{proof}

\subsection{Case 3: $\xi\in V, \ X, Y \in H$. Integrability condition(\ref{intexih11})}\hfill

In this case \eqref{integr} reads
\begin{multline}  \label{intexih1}
\nabla^3u(\xi_{i},X,Y)-\nabla^3u(X,\xi_{i},Y)=-R(\xi_{i},X,Y,du)-\nabla^2u(T(\xi_{i},X),Y).\hfill
\end{multline}
When we take a covariant
derivative along a Reeb vector field of \eqref{sist1} and a
covariant derivative along a horizontal direction of \eqref{add1}, use %
\eqref{add1}, \eqref{sist1}, \eqref{add2}, \eqref{add3}, \eqref{add4} and %
\eqref{comutat}, we obtain 
\begin{multline}\label{raitg}
\nabla ^3u(\xi _{i,}X,Y)-\nabla ^3u(X,\xi _{i},Y)+\nabla^2u(T(\xi
_{i},X),Y) \\
=-du(I_{i}Y)\left[\epsilon_i \mathbb{B}(I_{i}X,\xi
_{i})+\frac{1}{4n}d(tr\,L)(X)\right]
\,-du(I_{j}Y)\left[\epsilon_j \mathbb{B}(I_{j}X,\xi _{i})+\epsilon_i\mathbb{B}(X,\xi _{k})%
\right] \\
-du(I_{k}Y)\left[\epsilon_k \mathbb{B}(I_{k}X,\xi _{i})-\epsilon_i\mathbb{B}(X,\xi _{j})%
\right]  +g(X,Y)\mathbb{B}(du,\xi _{i})\\
+\epsilon_i\omega _{i}(X,Y)\mathbb{B}(I_{i}du,\xi
_{i})+\epsilon_j\omega _{j}(X,Y)\mathbb{B}(I_{j}du,\xi _{i})+\epsilon_k\omega _{k}(X,Y)\mathbb{B}%
(I_{k}du,\xi _{i}) \\
-du(X)\mathbb{B}(Y,\xi _{i})-\epsilon_idu(I_{i}X)\mathbb{B}(I_{i}Y,\xi _{i})-\epsilon_jdu(I_{j}X)%
\mathbb{B}(I_{j}Y,\xi _{i})-\epsilon_kdu(I_{k}X)\mathbb{B}(I_{k}Y,\xi _{i}) \\
+\frac{1}{4}\Big[(\nabla_XL)(Y,I_{i}du)-(\nabla_XL)(I_{i}Y,du)+\epsilon_i(\nabla_XL)(I_{k}Y,I_{j}du)-\epsilon_i(\nabla_XL)(I_{j}Y,I_{k}du)\Big] \\
-(\nabla_{\xi _{i}}L)(X,Y)\ \ -(\nabla_X\mathbb{B})(Y,\xi
_{i})+L(X,I_{i}LY)-T(\xi _{i},X,LY)-T(\xi _{i},Y,LX)\\
-\epsilon_i\omega _{i}(X,Y)\mathbb{%
B}(\xi _{i},\xi _{i})-\epsilon_j\omega _{j}(X,Y)\mathbb{B}(\xi _{i},\xi
_{j})-\epsilon_k\omega _{k}(X,Y)\mathbb{B}(\xi _{i},\xi _{k}).
\end{multline}
 On the other hand,  a substitution of \eqref{t01} and \eqref{u01} in \eqref{vert1}, an application of \eqref{verf1}
 together with the already proven \eqref{inte} and \eqref{inte1}  shows after a series of standard calculations that
\begin{multline}\label{rverg}
R(\xi _{i},X,Y,Z)=-\epsilon_j\mathbb{B}(I_{j}Z,\xi _{i})\omega _{j}(X,Y)-\epsilon_k\mathbb{B}(I_{k}Z,\xi
_{i})\omega _{k}(X,Y) \\
+\omega _{i}(Y,Z)\left[\epsilon_i \mathbb{B}(I_{i}X,\xi
_{i})+\frac{1}{4n}d(trL)(X) \right] +\omega _{j}(Y,Z)\Big[
\epsilon_i\mathbb{B}(X,\xi _{k})+\epsilon_j\mathbb{B}(I_{j}X,\xi _{i})\Big]\\ -\omega
_{k}(Y,Z)\Big[\epsilon_i \mathbb{B}(X,\xi _{j})-\epsilon_k\mathbb{B} (I_{k}X,\xi
_{i})\Big]   +g(X,Z)\mathbb{B}(Y,\xi _{i})+\epsilon_i\omega
_{i}(X,Z)\mathbb{B}(I_{i}Y,\xi _{i})\\
+\epsilon_j\omega _{j}(X,Z)\mathbb{B}(I_{j}Y,\xi _{i})+\epsilon_k\omega
_{k}(X,Z)\mathbb{B} (I_{k}Y,\xi _{i}) -g(X,Y)\mathbb{B}(Z,\xi
_{i})-\epsilon_i\mathbb{B}(I_{i}Z,\xi _{i})\omega _{i}(X,Y)\\
+\frac{1}{4}\Big[ (\nabla_XL)(I_{i}Y,Z)-(\nabla_XL)(Y,I_{i}Z)-\epsilon_i(
\nabla_XL)(I_{k}Y,I_{j}Z)+\epsilon_i(\nabla_XL)(I_{j}Y,I_{k}Z)\Big].
\end{multline}
In the derivation of the above equation we used   \eqref{bes} and Lemma~\ref{bijn=1}.

After substituting  equations \eqref{rverg} taken  with
$Z=du$ and \eqref{raitg} into \eqref{intexih1}, we obtain 
that the integrability condition \eqref{intexih1} reduces to
\begin{multline}  \label{intexih11}
(\nabla_{\xi_i}L)(X,Y)+ (\nabla_X\mathbb{B})(Y,\xi_i)+L(Y,I_iL(X))+L(T(%
\xi_i,X),Y)+g(T(\xi_i,Y),L(X)) \\
=- \sum_{s=1}^3\epsilon_s\mathbb{B}(\xi_s,\xi_i)\omega_s(X,Y).
\end{multline}
Notice that Case 3 implies Case 2 since 
\eqref{inte1} is  the skew-symmetric part of 
\eqref{intexih11}.


\begin{lemma}\label{vertric-ta}
For the vertical part of the Ricci 2-forms we have the equalities
\begin{equation}\label{ric-ta-v1}
\begin{aligned}
\rho_i(\xi_j,\xi_k)& =\frac1{8n^2}(tr\,L)^2+\epsilon_j\mathbb B(\xi_j,\xi_j)+\epsilon_k\mathbb B(\xi_k,\xi_k),\\
\rho_i(\xi_i,\xi_j)& =-\epsilon_i\frac1{4n}d(tr\,L)(\xi_j)-\epsilon_k\mathbb
B(\xi_i,\xi_k), \qquad \rho_i(\xi_i,\xi_k)&
=-\epsilon_i\frac1{4n}d(tr\,L)(\xi_k)+\epsilon_j\mathbb B(\xi_i,\xi_j).
\end{aligned}
\end{equation}
\end{lemma}
\begin{proof}
From the formula for the curvature \eqref{vert2} and
Proposition~\ref{torb} it follows
\begin{align*}
 4n\rho _{i}(\xi _{i},\xi _{k}) & =-\epsilon_k(\nabla_{e_a}\rho_{j})(I_{j}e_{a},\xi _{k})+
 T(\xi _{i},e_{a},e_{b})T(\xi
_{k},e_{b},I_{i}e_{a})
-T(\xi_{i},e_{b},I_{i}e_{a})T(\xi _{k},e_{a},e_{b});\\
4n\rho _{j}(\xi _{i},\xi _{k})&
=-(\nabla_{e_a}\rho_{j})(I_{i}e_{a},\xi _{k})+
T(\xi _{i},e_{a},e_{b})T(\xi _{k},e_{b},I_{j}e_{a}) -T(\xi
_{i},e_{a},I_{j}e_{b})T(\xi _{k},e_{b},e_{a}).
\end{align*}
Lemma \ref{l:rho-ta} allows us to compute 
\begin{eqnarray*}
(\nabla_{e_a}\rho _{i})(I_{k}e_{a},\xi _{j})=\epsilon_k
(\nabla_{e_a}\mathbb B)(I_{k}e_{a},\xi _{k})+\epsilon_j(\nabla_{e_a}\mathbb B)(I_{j}e_{a},\xi _{j}). 
\end{eqnarray*}%
After a calculation in which we use the integrability condition
\eqref{inte1}, the preceding paragraphs imply the first equation of
\eqref{ric-ta-v1}. 

For the calculation of $\rho _{i}(\xi _{i},\xi
_{k})$ we use again \eqref{inte1} to obtain
\begin{equation*}
(\nabla_{e_a}\mathbb{B})(I_{i}e_{a},\xi _{k})
=
-L(I_ie_a,I_ke_b)L(e_a,e_b)+4n \mathbb{B}(\xi _{i},\xi _{k}).
\end{equation*}%
Letting $A=\xi_{i}, B=X, C=Y, D=e_a, E=I_se_a$ in the second Bianchi
identity \eqref{secb} we get
\begin{multline}\label{nr1}
(\nabla_{\xi_{i}}\rho_s)(X,Y)-(\nabla_X\rho_s)(\xi_{i},Y)+(\nabla_Y\rho_s)(\xi_{i},X)\\
+\rho_s(T(\xi_{i},X),Y)-\rho_s(T(\xi_{i},Y),X)-2\sum_{t=1}^3\epsilon_t\omega_t(X,Y)\rho_s(\xi_t,\xi_{i})=0.
\end{multline}
Setting $s=i, Y=I_{i}X$ in \eqref{nr1}, using \eqref{ricis},
\eqref{symdh} with respect to the function $tr\,L$, together  with
Lemma~\ref{l:rho-ta} we  obtain
\begin{multline}\label{nrver1}
\epsilon_i\Big[(\nabla_{\xi_{i}}L)(X,X)+(\nabla_X\mathbb B)(X,\xi_{i})\Big] -
\Big[(\nabla_{\xi_{i}}L)(I_{i}X,I_{i}X)+(\nabla_{I_{i}X}\mathbb
B)(I_{i}X,\xi_{i})\Big]=\\\rho_{i}(e_a,X)\Big[T(\xi_{i},I_{i}X,e_a)-
T(\xi_{i},X,I_{i}e_a)\Big].
\end{multline}
We take the trace in \eqref{nrver1} and use the properties of the
torsion listed in Proposition~\ref{torb} to conclude
\begin{equation}\label{novo}
2\Big[(\nabla_{e_b}\mathbb{B})(e_{b},\xi _{i})+d(trL)(\xi
_{i})\Big]=\rho_{i}(e_a,e_b)\Big[T(\xi_{i},I_{i}e_b,e_a)-
T(\xi_{i},e_b,I_{i}e_a)\Big]=0,
\end{equation}
which implies the formula for $\rho _{i}(\xi _{i},\xi _{k})$ after a
short computation. Finally, with the help of \eqref{ricvert1}
we obtain the formula for $\rho _{i}(\xi _{i},\xi_{j})$.
\end{proof}
\begin{prop}\label{integrmain3}
If $W^{qc}=0$  then
the condition \eqref{intexih11} holds.
\end{prop}
\begin{proof}
It is sufficient to consider only the symmetric part of
\eqref{intexih11} since its skew-symmetric part is the already
established \eqref{inte1}.

Setting $s=j, Y=I_{j}X$ in \eqref{nr1}, using \eqref{ricis},
Lemma~\ref{l:rho-ta}, Lemma~\ref{vertric-ta} and \eqref{inte1}, we
calculate
\begin{multline}\label{nrver2}
\Big[(\nabla_{\xi_{i}}L)(X,X)+(\nabla_X\mathbb B)(X,\xi_{i})\Big] -\epsilon_j
\Big[(\nabla_{\xi_{i}}L)(I_{j}X,I_{j}X)+(\nabla_{I_{j}X}\mathbb
B)(I_{j}X,\xi_{i})\Big]
\\
=-2\epsilon_iL(X,I_{k}e_a)L(I_{j}X,e_a)-\epsilon_j\rho_{j}(e_a,X)\Big[T(\xi_{i},X,I_{j}e_a)-
T(\xi_{i},I_{j}X,e_a)\Big].
\end{multline}
Similarly, we  take $s=k, Y=I_{k}X$ in \eqref{nr1}, use
\eqref{ricis}, Lemma~\ref{l:rho-ta}, Lemma~\ref{vertric-ta} and
\eqref{inte1} to obtain
\begin{multline*}
\Big[(\nabla_{\xi_{i}}L)(X,X)+(\nabla_X\mathbb B)(X,\xi_{i})\Big] -\epsilon_k
\Big[(\nabla_{\xi_{i}}L)(I_{k}X,I_{k}X)+(\nabla_{I_{k}X}\mathbb
B)(I_{k}X,\xi_{i})\Big]
\\
=-2\epsilon_iL(I_{k}X,I_{j}e_a)L(X,e_a)-\epsilon_k\rho_{k}(e_a,X)\Big[T(\xi_{i},X,I_{k}e_a)-
T(\xi_{i},I_{k}X,e_a)\Big].
\end{multline*}
We replace $X$ with $I_{i}X$, 
subtract the
obtained equality from \eqref{nrver2} and add the result to
\eqref{nrver1} to get
\begin{multline}\label{nrver4}
2\Big[(\nabla_{\xi_{i}}L)(X,X)+(\nabla_X\mathbb B)(X,\xi_{i})\Big]
=-2\epsilon_iL(X,I_{k}e_a)L(I_{j}X,e_a)+2\epsilon_jL(I_{j}X,I_{j}e_a)L(I_{i}X,e_a)\\+\epsilon_i\rho_{i}(e_a,X)\Big[T(\xi_{i},I_{i}X,e_a)-
T(\xi_{i},X,I_{i}e_a)\Big]-\epsilon_j\rho_{j}(e_a,X)\Big[T(\xi_{i},X,I_{j}e_a)-
T(\xi_{i},I_{j}X,e_a)\Big]\\+\rho_{k}(e_a,I_{i}X)\Big[T(\xi_{i},I_{j}X,e_a)+\epsilon_j
T(\xi_{i},I_{i}X,I_{k}e_a)\Big].
\end{multline}
Next, using \eqref{t1} and the second equality in \eqref{ricis}
applied to \eqref{nrver4} together with some standard
calculations concludes the  proof of  \eqref{intexih11}.
\end{proof}

\subsection{Cases 4 and 5: $\xi_{i}, \xi_{j} \in V, X\in H$.
 Integrability conditions ~(\ref{intehxi312}), (\ref{intehxi31n}) and
(\ref{intehxi2})}\hfill

\textbf{Case 4: $\xi_{i}, \xi_{j} \in V,\quad X\in H$}. In this case
\eqref{integr} reads
\begin{multline}  \label{intehxi2}
\nabla^3u(\xi_{i},\xi_{j},X)-\nabla^3u(\xi_{j},\xi_{i},X)=-R(\xi_{i},\xi_{j},X,du)-\nabla^2u(T(\xi_{i},\xi_{j}),X).\hfill
\end{multline}
Working as in the previous case, using \eqref{add2},\eqref{add3},
\eqref{add4}, substituting \eqref{t01}, \eqref{u01}, \eqref{ricis}
into \eqref{vert2}, 
applying the already proven \eqref{inte}, \eqref{inte1},
\eqref{intexih11} and performing a series of standard calculations we conclude that \eqref{intehxi2} is equivalent to
\begin{multline}  \label{intehxi21}
(\nabla_{\xi_{i}}\mathbb{B})(X,\xi_{j})-(\nabla_{\xi_{j}}\mathbb{B}%
)(X,\xi_{i})=L(X,I_{j}e_a)\mathbb{B}(e_a,\xi_{i})-
L(X,I_{i}e_a)\mathbb{B}(e_a,\xi_{j})
\\
+\epsilon_jL(e_a,X)\rho_{k}(I_{i}e_a,\xi_{i})-T(\xi_{i},X,e_a)\mathbb{B}(e_a,\xi_{j})+T(%
\xi_{j},X,e_a)\mathbb{B}(e_a,\xi_{i})-\epsilon_k \frac{tr L}n\,\mathbb{B}(X,\xi_{k})\\
=\Big[2L(X,I_{j}e_a)+T(\xi_{j},X,e_a)\Big]\mathbb{B}(e_a,\xi_{i})-
\Big[2L(X,I_{i}e_a)+T(\xi_{i},X,e_a)\Big] \mathbb{B}(e_a,\xi_{j})-\epsilon_k\frac{tr L}n\,\mathbb{B}(X,\xi_{k}).
\end{multline}
where we used Lemma~\ref{l:rho-ta} to derive the second equality.

\textbf{Case $5_a$: $X \in H,\quad \xi_{i},\xi_{j}\in V$}. In this
case \eqref{integr} becomes
\begin{multline}  \label{intehxi3}
\nabla^3u(X,\xi_{i},\xi_{j})-\nabla^3u(\xi_{i},X,\xi_{j})
=-R(X,\xi_{i},\xi_{j},du)+\nabla^2u(T(\xi_{i},X),\xi_{j})= \\
-2\epsilon_jdu(\xi_{i})\rho_{k}(X,\xi_{i})+2\epsilon_jdu(\xi_{k})\rho_{i}(X,\xi_{i})+T(\xi_{i},X,e_a)\nabla^2
u(e_a,\xi_{j}).
\end{multline}
With  similar calculations as in the previous cases, we see that
\eqref{intehxi3} is equivalent to
\begin{multline}  \label{intehxi31n}
(\nabla_{\xi_{i}}\mathbb{B})(X,\xi_{j})+(\nabla_X\mathbb{B})(\xi_{i},\xi_{j}) \\
-2L(X,I_{j}e_a)\mathbb{B}(e_a,\xi_{i})+T(\xi_{i},X,e_a)\mathbb{B}%
(e_a,\xi_{j})+\epsilon_k\frac{tr L}{2n}\,\mathbb{B}(X,\xi_{k})=0.
\end{multline}


\textbf{Case $5_b$: $X \in H,\quad \xi_{j},\xi_{j}\in V$}. In this
case \eqref{integr} reads
\begin{multline}  \label{intehxi32}
\nabla^3u(X,\xi_{j},\xi_{j})-\nabla^3u(\xi_{j},X,\xi_{j})=-R(X,\xi_{j},\xi_{j},du)+\nabla^2u(T(\xi_{j},X),\xi_{j})= \\
-2\epsilon_jdu(\xi_{i})\rho_{k}(X,\xi_{j})+2\epsilon_jdu(\xi_{k})\rho_{i}(X,\xi_{j})+T(\xi_{j},X,e_a)\nabla^2u(e_a,\xi_{j}).
\end{multline}
and \eqref{intehxi32} is equivalent to
\begin{equation}  \label{intehxi312}
(\nabla_{\xi_{j}}\mathbb{B})(X,\xi_{j})+(\nabla_XB)(\xi_{j},\xi_{j})-2\mathbb{B}(e_a,\xi_{j})L(X,I_{j}e_a)+T(\xi_{j},X,e_a)\mathbb{B}(e_a,\xi_{j})=0.
\end{equation}
\begin{prop}
\label{integrmain451} If $W^{qc}=0$  then
the conditions \eqref{intehxi312}, \eqref{intehxi31n} and \eqref{intehxi2} hold.
\end{prop}
\begin{proof}
 Differentiating
the already proven \eqref{inte1} and taking the corresponding traces
we get
\begin{gather}\label{in321n=1}
(\nabla^2_{X,e_a}\mathbb{B})(I_ie_a,\xi_t)+2
(\nabla_XL)(e_a,e_b)L(I_ie_a,I_te_b)=4n(\nabla_X\mathbb{B})(\xi_i,\xi_t);\\
\label{in322n=1}
(\nabla^2_{e_a,X}\mathbb{B})(I_ie_a,\xi_t)-(\nabla^2_{e_a,I_ie_a}\mathbb{B}%
)(X,\xi_t) -2(\nabla_{e_b}L)(X,I_te_a)L(I_ie_b,e_a) 
\\\nonumber
-2(\nabla_{e_b}L)(I_ie_b,e_a)L(X,I_te_a)=2(\nabla_X\mathbb{B}%
)(\xi_i,\xi_t) +2\epsilon_i(\nabla_{I_kX}\mathbb{B})(\xi_j,\xi_t)-2\epsilon_i(\nabla_{I_jX}%
\mathbb{B})(\xi_k,\xi_t).
\end{gather}
Subtracting \eqref{in322n=1} from \eqref{in321n=1}, we obtain
\begin{multline}  \label{in323n=1}
\Bigl[\nabla^2_{X,e_a}-\nabla^2_{e_a,X}\Bigr]\mathbb{B}(I_ie_a,\xi_t)+(%
\nabla^2_{e_a,I_ie_a}\mathbb{B})(X,\xi_t) +2(\nabla_{e
_b}L)(I_ie_b,e_a)L(X,I_te_a) \\
+2\Bigl[(\nabla_XL)(e_a,e_b)-(\nabla_{e_b}L)(X,e_a)\Bigr]%
L(I_ie_b,I_te_a) \\
=2(2n-1)(\nabla_X\mathbb{B})(\xi_i,\xi_t) -2\epsilon_i(\nabla_{I_kX}\mathbb{B}%
)(\xi_j,\xi_t)+2\epsilon_i(\nabla_{I_jX}\mathbb{B})(\xi_k,\xi_t).
\end{multline}
The Ricci identities and \eqref{rhov} yield
\begin{gather}\label{in324n=1}
\Bigl[\nabla^2_{X,e_a}-\nabla^2_{e_a,X}\Bigr]\mathbb{B}(I_ie_a,\xi_{i}) =2(\nabla_{\xi_i}\mathbb{B})(X,\xi_{i})+2\epsilon_i(\nabla_{\xi_j}\mathbb{B}%
)(I_kX,\xi_{i}) -2\epsilon_i(\nabla_{\xi_k}\mathbb{B})(I_jX,\xi_{i})\\\nonumber
-4n\zeta_i(X,e_a)\mathbb{B}(e_a,\xi_{i})+2\epsilon_i\rho_{k}(X,e_a)\mathbb{B}%
(I_ie_a,\xi_{j})-2\epsilon_i\rho_{j}(X,e_a)\mathbb{B}(I_ie_a,\xi_{k}).
\\\label{in325n=1}
(\nabla^2_{e_a,I_{i}e_a}\mathbb{B})(X,\xi_{i})
=-2n\varrho_{i}(X,e_a)\mathbb{B}(e_a,\xi_{i})-4n(\nabla_{\xi_{i}}\mathbb{B}%
)(X,\xi_{i}).
\end{gather}
Next, we apply the already established \eqref{inte} and use the condition $%
L(e_a,I_se_a)=0$ to get
\begin{gather}\label{in326n=1}
\Bigl[(\nabla_XL)(e_a,e_b)-(\nabla_{e_b}L)(X,e_a)\Bigr] 
L(I_{i}e_b,I_{i}e_a) = 
-3\mathbb{B}(e_a,\xi_{i})L(X,I_{i}e_a)\\\nonumber- 3\epsilon_i\mathbb{B}(e_a,\xi_{j})L(I_{k}X,I_{i}e_a)+3\epsilon_i
\mathbb{B}(e_a,\xi_{k})L(I_{j}X,I_{i}e_a).
\\\label{in327n=1}
(\nabla_{e_b}L)(I_{i}e_b,e_a)L(X,I_{i}e_a)
=(4n+1)\mathbb{B}(e_a,\xi_{i})L(X,I_{i}e_a) +\epsilon_k\mathbb{B}(e_a,\xi_{j})L(X,I_{j}e_a)\\\nonumber+\epsilon_j
\mathbb{B}(e_a,\xi_{k})L(X,I_{k}e_a).
\end{gather}
We substitute \eqref{in327n=1}, \eqref{in326n=1}, \eqref{in325n=1}, \eqref{in324n=1} into %
\eqref{in323n=1} and  obtain after some calculations that
\begin{multline}  \label{int328n=1}
(1-2n)\Bigl[(\nabla_{\xi_{i}}\mathbb{B})(X,\xi_{i})+(\nabla_X\mathbb{B}%
)(\xi_{i},\xi_{i})\Bigr]+\epsilon_i \Bigl[(\nabla_{\xi_{j}}\mathbb{B})(I_{k}X,\xi_{i})+(%
\nabla_{I_{k}X}\mathbb{B})(\xi_{j},\xi_{i})\Bigr] \\
-\epsilon_i\Bigl[(\nabla_{\xi_{k}}\mathbb{B})(I_{j}X,\xi_{i})+(\nabla_{I_{j}X}\mathbb{B}%
)(\xi_{k},\xi_{i})\Bigr] = D_{ijk}(X),
\end{multline}
where $D_{ijk}(X)$ is defined by
\begin{multline}  \label{dijkn=1}
D_{ijk} (X)=\Bigl[2n\zeta_{i}(X,e_a)+n\varrho_{i}(X,e_a)-(4n-2)L(X,I_{i}e_a)\Bigr]\mathbb{%
B}(e_a,\xi_{i}) \\
+\epsilon_i \Bigl[\rho_{k}(X,I_{i}e_a)+3L(I_{k}X,I_{i}e_a)+\epsilon_jL(X,I_{j}e_a)\Bigr]\mathbb{B}%
(e_a,\xi_{j}) \\
-\epsilon_i \Bigl[\rho_{j}(X,I_{i}e_a)+3L(I_{j}X,I_{i}e_a)-\epsilon_kL(X,I_{k}e_a)\Bigr]\mathbb{B}%
(e_a,\xi_{k}).
\end{multline}
The second Bianchi identity \eqref{secb} taken with respect to
$A=\xi_{i}, B=\xi_{j},C=X, D=e_a, E=I_se_a$ and the formulas
described in Theorem~\ref{sixtyseven} yield
\begin{multline}\label{rhon=1}
(\nabla_{\xi_{i}}\rho_s)(\xi_{j},X)-(\nabla_{\xi_{j}}\rho_s)(\xi_{i},X)+(\nabla_X\rho_s)(\xi_{i},\xi_{j})\\
=\rho_s(T(\xi_{i},X),\xi_{j})-\rho_s(T(\xi_{j},X),\xi_{i})-\epsilon_j\rho_s(e_a,X)\rho_{k}(I_{i}e_a.\xi_{i})-\epsilon_k\frac{tr\,L}n\rho_s(\xi_{k},X).
\end{multline}
Setting   $s=1,2,3$  in \eqref{rhon=1}, using
\eqref{comutat} with respect to the function $tr\,L$ and applying
Lemma~\ref{l:rho-ta} and Lemma~\ref{vertric-ta}, we obtain after
some calculations
\begin{equation}\label{n=123}
\begin{aligned}
\Big[(\nabla_{\xi_{i}}\mathbb
B)(I_{i}X,\xi_{j})-(\nabla_{\xi_{j}}\mathbb B)(I_{i}X,\xi_{i})\Big]
+\epsilon_k\Big[(\nabla_{\xi_{i}}\mathbb B)(X,\xi_{k})+(\nabla_X\mathbb B)(\xi_{i},\xi_{k})\Big]=\alpha_{ijk}(X);\\
\Big[(\nabla_{\xi_{i}}\mathbb
B)(I_{j}X,\xi_{j})-(\nabla_{\xi_{j}}\mathbb B)(I_{j}X,\xi_{i})\Big]
+\epsilon_k\Big[(\nabla_{\xi_{j}}\mathbb B)(X,\xi_{k})+(\nabla_X\mathbb B)(\xi_{j},\xi_{k})\Big]=\beta_{ijk}(X);\\
\Big[(\nabla_{\xi_{i}}\mathbb
B)(I_{k}X,\xi_{j})-(\nabla_{\xi_{j}}\mathbb B)(I_{k}X,\xi_{i})\Big]
-\epsilon_i\Big[(\nabla_{\xi_{i}}\mathbb B)(X,\xi_{i})+(\nabla_X\mathbb B)(\xi_{i},\xi_{i})\Big]\qquad \qquad\quad \\
-\epsilon_j\Big[(\nabla_{\xi_{j}}\mathbb B)(X,\xi_{j})+(\nabla_X\mathbb
B)(\xi_{j},\xi_{j})\Big]=\gamma_{ijk}(X),
\end{aligned}
\end{equation}
where

\begin{equation}\label{alpha}
\begin{aligned}
\alpha_{ijk}(X)=\rho_{i}(e_a,\xi_{i})T(\xi_{j},X,e_a)& -\rho_{i}(e_a,\xi_{j})T(\xi_{i},X,e_a)+\epsilon_j\rho_{i}(e_a,X)\rho_{k}(I_{i}e_a,\xi_{i})\\
& -\epsilon_i\frac1{4n}d(tr\,L)(e_a)T(\xi_{j},X,e_a)+\epsilon_k\frac{tr\,L}{n}\rho_{i}(\xi_{k},X);\\
\beta_{ijk}(X)=\rho_{j}(e_a,\xi_{i})T(\xi_{j},X,e_a)& -\rho_{j}(e_a,\xi_{j})T(\xi_{i},X,e_a)+\epsilon_j\rho_{j}(e_a,X)\rho_{k}(I_{i}e_a,\xi_{i})\\
& +\epsilon_j\frac1{4n}d(tr\,L)(e_a)T(\xi_{i},X,e_a)+\epsilon_k\frac{tr\,L}{n}\rho_{j}(\xi_{k},X);\\
\gamma_{ijk}(X)=\rho_{k}(e_a,\xi_{i})T(\xi_{j},X,e_a)& -\rho_{k}(e_a,\xi_{j})T(\xi_{i},X,e_a)+\epsilon_j\rho_{k}(e_a,X)\rho_{k}(I_{i}e_a,\xi_{i})\\
& +\frac1{4n^2}(tr\,L)d(tr\,L)(X)+\epsilon_k\frac{tr\,L}{n}\rho_{k}(\xi_{k},X).
\end{aligned}
\end{equation}
Now, we solve the system consisting of \eqref{dijkn=1} and
\eqref{n=123} 
 with the help of Lemma~\ref{bijn=1}. We obtain 
\begin{multline}\label{nn=13}
2n\Big[(\nabla_{\xi_{i}}\mathbb
B)(I_{k}X,\xi_{j})-(\nabla_{\xi_{j}}\mathbb B)(I_{k}X,\xi_{i})\Big]
+\Big[(\nabla_{\xi_{k}}\mathbb B)(I_{j}X,\xi_{i})+(\nabla_{I_{j}X}\mathbb B)(\xi_{k},\xi_{i})\Big]\\
-\Big[(\nabla_{\xi_{k}}\mathbb
B)(I_{i}X,\xi_{j})+(\nabla_{I_{i}X}\mathbb
B)(\xi_{k},\xi_{j})\Big]=-\epsilon_iD_{ijk}(X)-\epsilon_jD_{jki}(X)+(2n-1)\gamma_{ijk}(X).
\end{multline}
The first two equalities in \eqref{n=123} together with \eqref{nn=13} lead to
\begin{multline}\label{nn=14}
2(n+1)\Big[(\nabla_{\xi_{i}}\mathbb
B)(I_{k}X,\xi_{j})-(\nabla_{\xi_{j}}\mathbb B)(I_{k}X,\xi_{i})\Big]
+\Big[(\nabla_{\xi_{k}}\mathbb
B)(I_{j}X,\xi_{i})-(\nabla_{\xi_{i}}\mathbb B)(I_{j}X,\xi_{k})\Big]
+\\\Big[(\nabla_{\xi_{j}}\mathbb
B)(I_{i}X,\xi_{k})-(\nabla_{\xi_{k}}\mathbb
B)(I_{i}X,\xi_{j})\Big]=A_{ijk}(X),
\end{multline}
where
\begin{equation}\label{nn=15}
A_{ijk}(X)=-\epsilon_iD_{ijk}(X)-\epsilon_jD_{jki}(X)+(2n-1)\gamma_{ijk}(X)-\epsilon_k\alpha_{ijk}(I_{j}X)+\epsilon_k\beta_{ijk}(I_{i}X).
\end{equation}
Consequently, we derive easily that
\begin{multline}\label{nn=16}
2(n+2)(2n+1)\Big[(\nabla_{\xi_{i}}\mathbb
B)(I_{k}X,\xi_{j})-(\nabla_{\xi_{j}}\mathbb
B)(I_{k}X,\xi_{i})\Big]\\=(2n+3)A_{ijk}(X)- A_{jki}(X)-A_{kij}(X).
\end{multline}
The second equality in \eqref{ricis} together with \eqref{t1} and Lemma~\ref{l:rho-ta} applied to \eqref{alpha} and \eqref{dijkn=1}, followed by some standard calculations give expressions of 
$$\epsilon_k\beta_{ijk}(I_iX)-\epsilon_k\alpha_{ijk}(I_jX) \quad and \quad -\epsilon_iD_{ijk}(X)-\epsilon_jD_{jki}(X)+(2n-1)\gamma_{ijk}(X)$$
 in terms of $L$ and $B(.,\xi)$, which substituted into \eqref{nn=15} yield  the following formula for $A_{ijk}(X)$
\begin{multline}\label{nn=110}
A_{ijk}(X)=-\frac{1}{4n}\left( tr\,L\right) \Big[
(1-2n)\epsilon_i\,\mathbb{B}(I_{i}X,\xi
_{i})+\left( 1-2n\,\right)\epsilon_j \mathbb{B}(I_{j}X,\xi _{j})+\left( 8n+6\right)\epsilon_k\,%
\mathbb{B}(I_{k}X,\xi _{k})\Big]  \\
-\frac{1}{4}L(X,e_{a})\Big[ \left( 2n+3\right)\epsilon_i B(I_{i}e_{a},\xi
_{i})+\left( 2n+3\right)\epsilon_j B(I_{j}e_{a},\xi _{j})+2\epsilon_kB(I_{k}e_{a},\xi _{k})%
\Big]  \\
+\frac{1}{4}L(I_{i}X,e_{a})\Big[ \left( 2n+3\right) \epsilon_iB(e_{a},\xi
_{i})+\left( 2n-3\right) B(I_{k}e_{a},\xi _{j})+4B(I_{j}e_{a},\xi _{k})%
\Big]  \\
+\frac{1}{4}L(I_{j}X,e_{a})\Big[ -\left( 2n-3\right)
B(I_{k}e_{a},\xi
_{i})+\left( 2n+3\right)\epsilon_j B(e_{a},\xi _{j})-4B(I_{i}e_{a},\xi _{k})\Big]  \\
+\frac{1}{4}L(I_{k}X,e_{a})\Big[ -\left( 10n+9\right)
B(I_{j}e_{a},\xi _{i})+\left( 10n+9\right) B(I_{i}e_{a},\xi
_{j})+2\epsilon_kB(e_{a},\xi _{k})\Big].
\end{multline}
Inserting \eqref{nn=110} into \eqref{nn=16} and using \eqref{t1} we arrive at the proof of  \eqref{intehxi21}.  We substitute \eqref{intehxi21} into 
the first equality of \eqref{n=123} to obtain \eqref{intehxi31n}.
We insert \eqref{intehxi31n} into \eqref{int328n=1} to see that
\eqref{intehxi32} holds.
\end{proof}

\subsection{Case 6: $\xi_{k},\xi_{i},\xi_{j}\in V$. Integrability conditions ~(\ref{intehxi41}) and
(\ref{intehxi441})}\hfill

\textbf{Case $6_{a}$: $\xi_{k},\xi_{i},\xi_{j}\in V$}. In this case 
the Ricci identity  \eqref{integr} becomes
\begin{multline}  \label{intehxi4}
\nabla^3u(\xi_{k},\xi_{i},\xi_{j})-
\nabla^3u(\xi_{i},\xi_{k},\xi_{j})=-R(\xi_{k},\xi_{i},\xi_{j},du)-\nabla^2u(T(\xi_{k},\xi_{i}),\xi_{j}) \\
=-2\epsilon_jdu(\xi_{i})\rho_{k}(\xi_{k},\xi_{i})+2\epsilon_jdu(\xi_{k})\rho_{i}(\xi_{k},\xi_{i})-\epsilon_i\rho_{j}(I_{k}e_a,%
\xi_{k})\nabla^2u(e_a,\xi_{j}) -\frac{\epsilon_j}n( tr\,L)\,\nabla^2u(\xi_{j},\xi_{j}).
\end{multline}
With the help of Lemma \ref{l:rho-ta} and Lemma~\ref{vertric-ta} we see after some calculations that the
integrability condition \eqref{intehxi4} takes the form
\begin{multline}\label{intehxi41}
(\nabla_{\xi _{i}}\mathbb{B})(\xi _{k},\xi _{j})-(\nabla_{\xi _{k}}
\mathbb{B})(\xi _{i},\xi _{j}) =-\frac{1}{2n}(tr\,L\,)\left[\epsilon_i
\mathbb{B}(\xi _{i},\xi _{i})-2\epsilon_j\mathbb{B}(\xi _{j},\xi
_{j})+\epsilon_k\mathbb{B}(\xi _{k},\xi _{k}) \right]\\
+2\mathbb{B}(e_{a},\xi _{i})\mathbb{B}(I_{j}e_{a},\xi
_{k})+\mathbb{B} (e_{a},\xi _{i})\mathbb{B}(I_{k}e_{a},\xi
_{j})+\mathbb{B}(I_{i}e_{a},\xi _{k})\mathbb{B}(e_{a},\xi _{j}).
\end{multline}
\textbf{Case $6_{b}$: $\xi_{k},\xi_{j},\xi_{j}\in V$}. In this case,
equation \eqref{integr} reads
\begin{multline}  \label{intehxi44}
\nabla^3u(\xi_{k},\xi_{j},\xi_{j})-\nabla^3u(\xi_{j},\xi_{k},\xi_{j})=-R(\xi_{k},\xi_{j},\xi_{j},du)-\nabla^2u(T(\xi_{k},\xi_{j}),\xi_{j})\\=
-2\epsilon_jdu(\xi_{i})\rho_{k}(\xi_{k},\xi_{j})+2\epsilon_jdu(\xi_{k})\rho_{i}(\xi_{k},\xi_{j}) +\epsilon_j\rho_{i}(I_{k}e_a,%
\xi_{k})\nabla^2u(e_a,\xi_{j}) +\epsilon_i\frac{(tr\,L)}n\,\nabla^2u(\xi_{i},\xi_{j}).
\end{multline}
 After a series of straightforward calculations similar as above, we show that \eqref{intehxi44} is equivalent to
\begin{equation}\label{intehxi441}
(\nabla_{\xi _{j}} \mathbb{B})(\xi _{k},\xi _{j})-(\nabla_{\xi _{k}}
\mathbb{B})(\xi _{j},\xi _{j})=3\mathbb{B}(I_{j}e_{a},\xi _{k})\mathbb{B}(e_{a},\xi _{j})
-\frac{3\epsilon_i}{2n}(tr\,L)\, \mathbb{B}(\xi _{i},\xi _{j}).
\end{equation}
\begin{prop}
\label{integrmain6} If $W^{qc}=0$  then
the conditions \ref{intehxi41}, \ref{intehxi441} hold.
\end{prop}
\begin{proof}
Differentiate
\eqref{intehxi31n} and take the corresponding trace to get
\begin{multline}
(\nabla _{e_{a},\xi _{i}}^{2}\mathbb{B})(I_{k}e_{a},\xi
_{j})+(\nabla
_{e_{a},I_{k}e_{a}}^{2}\mathbb{B})(\xi _{i},\xi _{j})=  \label{int60} \\
2(\nabla _{e_{b}}L)(I_{k}e_{b},I_{j}e_{a})\mathbb{B}(e_{a},\xi
_{i})+2L(I_{k}e_{b},I_{j}e_{a})(\nabla _{e_{b}}\mathbb{B})(e_{a},\xi _{i}) \\
-(\nabla _{e_{b}}T)(\xi _{i},I_{k}e_{b},e_{a})\mathbb{B}(e_{a},\xi
_{j})-T(\xi _{i},I_{k}e_{b},e_{a})(\nabla
_{e_{b}}\mathbb{B})(e_{a},\xi _{j})
\\
-\frac{\epsilon_k}{2n}d(tr\,L)(e_{a})\mathbb{B}(I_{k}e_{a},\xi _{k})-\frac{\epsilon_k}{2n}%
(tr\,L)(\nabla _{e_{a}}\mathbb{B})(I_{k}e_{a},\xi _{k}).
\end{multline}%
On the other hand, the Ricci identities, \eqref{bst}, \eqref{rhov} and %
\eqref{ricis} yield
\begin{gather} \label{int601}
(\nabla _{e_{a},I_{k}e_{a}}^{2}\mathbb{B})(\xi _{i},\xi
_{j})=-4n(\nabla _{\xi _{k}}\mathbb{B})(\xi _{i},\xi
_{j})-4\epsilon_j(tr\,L)\mathbb{B}(\xi _{j},\xi _{j})+4\epsilon_i(tr\,L)\mathbb{B}(\xi
_{i},\xi _{i});\\ \label{int602}
(\nabla _{e_{a},\xi _{i}}^{2}\mathbb{B})(I_{k}e_{a},\xi
_{j})
=4n(\nabla _{\xi _{i}}\mathbb{B})(\xi _{j},\xi _{k})-2(\nabla _{\xi
_{i}}L)(e_{a},e_{b})L(I_{j}e_{a},I_{k}e_{b})+4n\zeta _{k}(\xi _{i},e_{a})%
\mathbb{B}(e_{a},\xi _{j}) \\\nonumber
+2\epsilon_j\rho _{i}(e_{a},\xi _{i})\mathbb{B}(I_{k}e_{a},\xi _{k})-2\epsilon_j\rho
_{k}(e_{a},\xi _{i})\mathbb{B}(I_{k}e_{a},\xi _{i})+T(\xi
_{i},e_{a},e_{b})(\nabla _{e_{b}}\mathbb{B})(I_{k}e_{a},\xi _{j})
\end{gather}
Substituting \eqref{int601} and \eqref{int602} in \eqref{int60} we obtain
\begin{multline}\label{e:c6.1}
4n\Big[(\nabla _{\xi _{i}}\mathbb{B})(\xi _{j},\xi _{k})-(\nabla _{\xi _{k}}%
\mathbb{B})(\xi _{i},\xi _{j})\Big]\\
= 2(\nabla _{e_{b}}L)(I_{k}e_{b},I_{j}e_{a})\mathbb{B}(e_{a},\xi
_{i}) + 2\Big[(\nabla _{e_{b}}\mathbb{B})(e_{a},\xi _{i})+(\nabla
_{\xi
_{i}}L)(e_{b},e_{a})\Big]L(I_{k}e_{b},I_{j}e_{a})\\-\mathbb{B}(e_{a},\xi
_{j}) \Big[4n\zeta _{k}(\xi _{i},e_{a})+(\nabla _{e_{b}}T)(\xi
_{i},I_{k}e_{b},e_{a})\Big]-\Big[2\epsilon_j\rho _{i}(e_{a},\xi
_{i})+\frac{\epsilon_k}{2n}d(tr\,L)(e_{a})\Big] \mathbb{B}(I_{k}e_{a},\xi
_{k})\\+2\epsilon_j\rho _{k}(e_{a},\xi _{i})\mathbb{B}(I_{k}e_{a},\xi _{i})
+T(\xi_{i},I_{k}e_{a},e_{b})\Big[(\nabla
_{e_{b}}\mathbb{B})(e_{a},\xi _{j})- (\nabla
_{e_{a}}\mathbb{B})(e_{b},\xi _{j})\Big]
\\
-\frac{\epsilon_k}{2n}%
(tr\,L)(\nabla _{e_{a}}\mathbb{B})(I_{k}e_{a},\xi _{k})+4\epsilon_j(tr\,L)\mathbb{B}%
(\xi _{j},\xi _{j})-4\epsilon_i(tr\,L)\mathbb{B}(\xi _{i},\xi _{i}).
\end{multline}%
 We find with the help
of \eqref{rverg}, the symmetry of $L$ and 
\eqref{bes} that
\begin{multline*}
4n\zeta _{k}(\xi _{i},e_{a})=(4n+1)\mathbb{B}(I_{k}e_{a},\xi
_{i})+\epsilon_j\mathbb{B}(e_{a},\xi _{j})+\mathbb{B }(I_{i}e_{a},\xi
_{k})-\frac{\epsilon_j}{4n}d(trL)(I_{j}e_{a})\\
 -\frac{1}{4}\Big[( \nabla_{e_{b}}
L)(I_{i}e_{a},I_{k}e_{b})-\epsilon_j(\nabla_{e_{b}}
L)(e_{a},I_{j}e_{b})+(\nabla_{e_b}
L)(I_{k}e_{a},I_{i}e_{b})-\epsilon_j(\nabla_{e_b}
L)(I_{j}e_{a},e_{b})\Big]  .
\end{multline*}

It follows from \eqref{t1} that
\begin{multline*}
(\nabla_{e_b}T)(\xi _{i},I_{k}e_{b},e_{a}) =-\frac14\Big[\epsilon_j(\nabla_{e_b}L)(I_{j}{ e_{b}},{e_{a}})
+3(\nabla_{e_b}L)({I_{k}e_{b}},I_{i}{e_{a}})+\epsilon_j (\nabla_{e_b}
L)({e_{b}},I_{j}e_{a})\\
-(\nabla_{e_b}L)(I_{i}e_{b},I_{k}e_{a})
-\frac{\epsilon_j}{n}d(trL)(I_{j}e_{a})\Big].
\end{multline*}
The sum of the last two equalities  yields
\begin{equation}\label{comput3}
{4n\zeta }_{k}{(\xi }_{i}{,e}_{a}{)+(\nabla_{e_b}T)(\xi }_{i}{,I}_{k}{e}_{b}{,e}_{a})
=4nB(I_{k}{e}_{a},\xi_i)-4nB(I_{i}{e}_{a},{\xi}_{k}).
\end{equation}
Finally, combining \eqref{e:c6.1}, \eqref{comput3}, Lemma \ref{l:rho-ta},
\eqref{intexih11}, \eqref{inte1} with 
$L(e_{b},I_{s}e_{b})=0$ leads to a series of standart calculations, which at the end imply  \eqref{intehxi41}.

The other integrability condition in this case, \eqref{intehxi441},
can be obtained similarly using \eqref{intehxi312} and the Ricci
identities. The proof of Theorem~\ref{main2} is completed.
\end{proof}

\end{document}